\documentclass{amsart}

\usepackage{amssymb, amsfonts, amsmath, amsthm,bm, amscd}
\usepackage[dvipsnames]{xcolor}
\usepackage{mathrsfs}
\usepackage{hyperref}
\usepackage{graphicx}
\graphicspath{ {./img/} }
\usepackage{caption}
\usepackage{subcaption}
\usepackage{wrapfig}
\usepackage[margin=1.00 in]{geometry}
\usepackage{enumitem, multicol}
\usepackage{tikz}
\usetikzlibrary{cd}
\usepackage{nicefrac, bigints}
\usepackage{array}
\usepackage[normalem]{ulem}
\usepackage{mathtools}

\theoremstyle{definition}
\newtheorem{theorem}{Theorem}[section]

\newtheorem{definition}[theorem]{Definition}
\newtheorem{lemma}[theorem]{Lemma}

\newtheorem{corollary}[theorem]{Corollary}
\newtheorem{proposition}[theorem]{Proposition}
\newtheorem{remark}[theorem]{Remark}

\newtheorem{claim}[theorem]{Claim}

\newtheorem{construction}[theorem]{Construction}
\newtheorem{example}[theorem]{Example}

\newcommand{\QQ}{\mathbb{Q}}
\newcommand{\RR}{\mathbb{R}}
\newcommand{\CC}{\mathbb{C}}

\newcommand{\GL}{\mathsf{GL}}

\renewcommand{\Re}{\operatorname{Re}}
\renewcommand{\Im}{\operatorname{Im}}

\newcommand*{\marking}{\varphi}

\newcommand*{\geomintersect}{\iota}
\newcommand*{\algintersect}{\hat{\iota}}
\newcommand*{\TV}{TV}
\newcommand*{\valpha}{\vec{\alpha}}
\newcommand*{\vbeta}{\vec{\beta}}
\newcommand*{\vgamma}{\vec{\gamma}}

\newcommand*{\anga}{\angle a}
\newcommand*{\angb}{\angle b}

\newcommand*{\va}{\vec{a}}
\newcommand{\vb}{\vec{b}}
\newcommand{\vc}{\vec{c}}

\newcommand{\vh}{\vec{h}}

\newcommand{\singletons}{\mathbb{S}}

\DeclareMathOperator{\hol}{hol}

\DeclarePairedDelimiter{\abs}{\lvert}{\rvert}

\newcommand{\Addresses}{{
  \bigskip
  \footnotesize
  \noindent Juliet Aygun, \textsc{Department of Mathematics, Cornell University.}\par\nopagebreak
  \textit{E-mail}: \texttt{ja742@cornell.edu}
  
  \noindent Janet Barkdoll, \textsc{Swarthmore College.}\par\nopagebreak
  \textit{E-mail}: \texttt{jbarkdo1@alum.swarthmore.edu}
  
  \noindent Aaron Calderon, \textsc{Department of Mathematics, The University of Chicago.}\par\nopagebreak
  \textit{E-mail}: \texttt{aaroncalderon@uchicago.edu}
  
  \noindent Jenavie Lorman.\par\nopagebreak
  \textit{E-mail}: \texttt{jenav.lor@gmail.com}
 
  \noindent Theodore Sandstrom, \textsc{Department of Mathematics, Statistics, and Computer Science, University of Illinois Chicago.}\par\nopagebreak
  \textit{E-mail}: \texttt{tsands3@uic.edu}
  }}

\begin{document}

\title[Realizing pairs of multicurves as cylinders]{Realizing pairs of multicurves as cylinders\\ on translation surfaces}

\author[Aygun]{Juliet Aygun}
\author[Barkdoll]{Janet Barkdoll}
\author[Calderon]{Aaron Calderon}
\author[Lorman]{Jenavie Lorman}
\author[Sandstrom]{Theodore Sandstrom}

\begin{abstract}
Any pair of intersecting cylinders on a translation surface is ``coherent,’’ in that 
the geometric and algebraic intersection numbers of their core curves are equal (up to sign).
In this paper, we investigate when a pair of multicurves can be simultaneously realized as the core curves of cylinders on some translation surface.
Our main tools are surface topology and the ``flat grafting’’ deformation introduced by Ser-Wei Fu.
\end{abstract}

\maketitle

\section{Introduction}\label{sec:introduction}

Let $S$ be a closed surface of genus $g$.
The structure of a {\em translation surface} is an atlas of charts on $S$ away from finitely many cone points whose transition functions are given by translations. Equivalently, a translation surface may be thought of as a collection of polygons in the plane with sides glued by translations (up to cut-and-paste equivalence) or as an identification of $S$ with a Riemann surface $X$ equipped with a holomorphic 1-form $\omega$.

Given a simple closed curve $c$ on $S$ and a translation structure $(X, \omega)$ on $S$, say that $c$ can be \textit{realized as a cylinder} on $(X, \omega)$ if it is isotopic to the core curve of some embedded Euclidean cylinder on $(X,\omega)$.
A {\em multicurve} $\gamma$ on $S$ is a union of pairwise disjoint simple closed curves; if each curve of $\gamma$ is realized as a cylinder on $(X, \omega)$, then we say that the union of these cylinders is a {\em multicylinder.}
A multicylinder is {\em parallel} if all of the core curves have the same slope (relative to the horizontal vector field).
Cylinders and multicylinders have long been an important tool in the study of translation surfaces (see, e.g., \cite{Wcyldefs}), and recently new attention has been paid to their surface-topological aspects
\cite{AMNorigami, CHANG2021107730, jeffreys2021single}.

Given any multicurve $\gamma$ on $S$, it is not hard to prove that $\gamma$ can be realized as a multicylinder on some translation surface if and only if each of its curves is non-separating: simply find a translation surface containing a multicylinder of the correct topological type (e.g., using \cite{ZorJS} or \cite[\S 6.3]{strata2}) and use a homeomorphism to take $\gamma$ to that multicylinder.
On the other hand, given a {\em pair} of multicurves $\alpha$ and $\beta$, their intersection pattern may obstruct the existence of a translation surface which simultaneously realizes them as a pair of multicylinders.
This is because, given an orientation of the multicurves, the signs of the intersection points constrain the slopes of the realizing cylinders.

This sort of reasoning implies that if two oriented curves $\va$ and $\vb$ are simultaneously realizable as cylinders on a translation surface, then their geometric intersection number must equal their algebraic intersection number (up to sign); see Lemma \ref{1 parallel implies coherently orientable} below.
Such a pair of curves is called {\em coherent},
and if $\valpha$ and $\vbeta$ are two oriented multicurves, we say they are {\em coherent} if the same property holds.
Observe that for multicurves, coherence is stronger than the property of {\em pairwise coherence,} which just ensures that each pair of curves $\va \subset \valpha$ and $\vb \subset \vbeta$ are coherent.
Given unoriented multicurves $\alpha$ and $\beta$, we say they are {\em coherently orientable} if they can be oriented to be coherent.

We recall that a pair of multicurves {\em fills} $S$ if the complement of their union is a collection of disks.
Filling pairs are important for a number of reasons, not least because they can be used to build explicit examples of pseudo-Anosov homeomorphisms \cite{PennerpA}.

\begin{theorem}\label{mainthm:parallel}
Let $(\alpha, \beta)$ be a pair of multicurves on $S$. The following are equivalent:
\begin{enumerate}
    \item There exists a translation surface on which $\alpha$ and $\beta$ are both realized as parallel multicylinders.
    \item There is a coherently orientable filling pair $(\alpha', \beta')$ with $\alpha' \supset \alpha$ and $\beta' \supset \beta$.
    \item The multicurves $\alpha$ and $\beta$ are coherently orientable {\em and} no curve of $\alpha$ separates $S \setminus \beta$ and no curve of $\beta$ separates $S \setminus \alpha$ (in particular, no curves of $\alpha$ and $\beta$ are isotopic).
\end{enumerate}
\end{theorem}

That (2) implies (1) is well known, and that (1) implies (2) is not hard (see Lemma \ref{lemma:tv-directional-equivalence}). Our main contribution, then, is showing the equivalence of these conditions with condition (3), which is easily checkable given two multicurves $\alpha$ and $\beta$.
In fact, we prove a stronger theorem in which the orientations of $\alpha$ and $\beta$ may be prescribed and where the core curves of the realizing multicylinders must all point the same direction. See Theorem \ref{theorem:exttv_oriented}.

Note that condition (3) implies that each curve in $\alpha \cup \beta$ must be non-separating.
Thus, as a special case of Theorem \ref{mainthm:parallel}, we have the following corollary.

\begin{corollary}\label{corollary_to_mainthm:exttv_unoriented}
Two non-separating curves $a$ and $b$ are jointly realizable as cylinders on a translation surface if and only if they are coherently orientable.
\end{corollary}

Two multicurves both being realized as parallel multicylinders is a very strong condition, and one could weaken what it means for $\alpha$ and $\beta$ to be realized.
Say that a triple of unoriented multicurves $(\gamma_1,\gamma_2,\gamma_3)$ is coherently orientable if we can assign orientations so that simultaneously every pair $(\vgamma_i,\vgamma_j)$ is coherent.

\begin{theorem}
\label{maintheorem:totally-coherent-realizable}
Let $(\alpha, \beta)$ be a pair of multicurves on $S$.
Then there exists a translation surface on which $\alpha$ is realizable as a parallel multicylinder and $\beta$ is realizable as an arbitrary multicylinder
if and only if the following conditions hold:
\begin{enumerate}
    \item There exists a filling pair $(\alpha',\gamma)$ such that $\alpha \subset \alpha'$ and $(\alpha',\gamma,b)$ is coherently orientable for each curve $b \subset \beta$.
    \item Each curve of $\alpha$ and of $\beta$ is non-separating.
\end{enumerate}
\end{theorem}

It remains an open question to give a characterization of condition (1) solely in terms of the combinatorial data of $\alpha$ and $\beta$ on $S$. In particular, is (1) equivalent to being coherently orientable?
\medskip

At the weakest end of the spectrum, one could drop all assumptions on parallelism and ask for $\alpha$ and $\beta$ to be realized as any sort of multicylinders on some translation surface. In this setting:
\begin{enumerate}
    \item Coherence is not necessary (Example \ref{ex:multicyl_notcoherent}), 
    \item Pairwise coherence is necessary, yet 
    \item Pairwise coherence is not sufficient (Example \ref{ex:pairwisecoherent_notrealizable}).
\end{enumerate}

A further open question is characterizing all the necessary and sufficient conditions to realize any pair of multicurves as arbitrary multicylinders on some translation surface.

\subsection{Outline}
In Section \ref{sec:background}, we collect some relevant background about cylinders on translation surfaces.
One of the most important definitions in this section is the ``Thurston--Veech construction'' that connects filling pairs of multicurves with translation surfaces.
In Section \ref{sec:coherent-filling-pairs-of-multicurves}, we prove Theorem \ref{mainthm:parallel} and its generalization, Theorem \ref{theorem:exttv_oriented}, using the Thurston--Veech construction and surface-topological arguments.
Theorem \ref{maintheorem:totally-coherent-realizable} is proven in Section \ref{sec:total-coherence-and-grafting} using an explicit surgery called ``horizontal grafting'' which deforms a translation surface enough to realize certain concatenations of saddle connections as cylinders.
Finally, in Section \ref{sec:pairwise-coherent-multicurves-as-cylinders}, we give a number of examples that show that the question of realizing a pair of multicurves as arbitrary multicylinders is more subtle than just (pairwise) coherence.

\subsection{Acknowledgements}
This project began at the Yale SUMRY REU program in summer 2021, and we would like to thank all of the participants of SUMRY for their support and attention to our project.
We would also like to thank Marissa Loving, Nick Salter, and Jane Wang for helpful conversations as well as the rest of the visitors to the SUMRY program for contributing ideas and enthusiasm.
AC acknowledges financial support from NSF grants DMS-2005328 and DMS-2202703.

\section{Cylinders on translation surfaces}\label{sec:background}
In this section, we discuss the basics of cylinders on translation surfaces.
We first record a number of useful properties concerning intersections and coherence between geodesics on a translation surface.
We then recall the Thurston--Veech construction, which realizes a filling pair of multicurves as multicylinders on a (half)-translation surface.
This section also contains a proof that if this filling pair is coherent, then the construction yields a translation surface.

\subsection{Translation surfaces}
A \textit{translation surface} is a collection of compact polygons in $\mathbb{C}$ in which each side (vector in $\mathbb{R}^2 \cong \mathbb{C}$) is identified with another by translation, considered up to cut-and-paste equivalence. Performing the gluings constructs a topological surface $S$, which inherits from $\CC$ a metric that is Euclidean everywhere except on a finite set of cone points.
Equivalently, a translation surface is a Riemann surface $X$ equipped with a holomorphic abelian differential $\omega$. The cone points of the flat metric correspond to the zeros of $\omega$. We denote a translation surface by the pair $(X,\omega)$.

The equivalence can be seen as follows. Let $\Sigma = P_1,...,P_n$ be the zeros of $\omega$.
A geodesic arc whose endpoints are in $\Sigma$ and is otherwise disjoint from $\Sigma$ in its interior is called a \textit{saddle connection}. 
For an oriented curve or arc $\va$ on $S$, we denote its \textit{period} on the translation surface $(X, \omega)$ by
\[\hol(\va):= \int_{\va} \omega.\]
equivalently, the period is the vector obtained by connecting the images of $\va(0)$ and $\va(1)$ under the developing map of the flat metric.
When $a$ is unoriented, its period is defined up to multiplication by $-1$.
  
Choose some basis $c_1,\ldots,c_{2g+n-1}$ of $H_1(X,\Sigma; \mathbb{Z})$ consisting of saddle connections. Cutting along the $c_i$'s, we obtain a collection of polygons where each of the two sides corresponding to one of the $c_i$ has direction vector given by the complex number $\hol(c_i)$.
Moreover, the vertices of these polygons glue up to be $\Sigma$, and the local horizontal vector field on the surface determined by $\omega$ corresponds to the horizontal vector field on $\mathbb{C}$.
We note that this construction is indeed independent on our choice of basis for $H_1(X,\Sigma, \mathbb{Z})$ because if two curves or arcs $\gamma_1$ and $\gamma_2$ are homologous, then $\hol(\gamma_1) = \hol(\gamma_2).$

A \emph{stratum} of translation surfaces is the subset of all translation surfaces with a fixed number and angle of cone points.
The {\em periods} of a translation surface are the numbers $\{\hol(c_i)\}$ defined above.
Adjusting the periods slightly corresponds to adjusting the edges of the polygons defining our translation surface, and doing this we still obtain obtain a translation surface in the same stratum.
In this way the periods of a translation surface establish local (orbifold) {\em period coordinates} for strata modeled by $H^1(X,\Sigma; \mathbb{C})$.

Similarly, a \textit{half-translation surface} is a collection of polygons in $\mathbb{C}$ in which each side is identified with another by translation with possibly a rotation by $\pi$, considered up to cut-and-paste equivalence. Equivalently, a half-translation surface is a Riemann surface $X$ equipped with a quadratic differential $q$. Sometimes, $q$ may be the square of an abelian differential (i.e., the sides of the polygon may all be glued by translation), so one can regard translation surfaces as a subset of half-translation surfaces.
Strata of quadratic differentials also have period coordinates, but since we will not use them, we omit this discussion.

We refer the reader to \cite{Zsurvey} for a more thorough background on (half-)translation surfaces, period coordinates, and strata.

\subsection{Geodesics and multicylinders}
Two curves on a surface are said to be in {\em minimal position} if they realize the geometric intersection number for their isotopy classes. Throughout the paper, unless otherwise stated, we will assume that all curves are realized in minimal position.

A(n isotopy class of a) curve $c$ on a (half-) translation surface $(X,\omega)$ is {\em realizable as a cylinder} if $c$ is isotopic to the core curve of some embedded Euclidean cylinder on $(X,\omega)$.
In other words, there is some geodesic representative of $c$ that has constant slope on $(X,\omega)$ and does not contact a cone point.
In fact, the family of geodesic representatives sweeps out the embedded cylinder, all have the same length, and all representatives except for the two boundary components of the cylinder are nonsingular.
The circumference of the embedded cylinder is equal to the magnitude of $\hol(c)$ (equipped with either orientation).

\begin{definition}\label{def:multicyl}
A {\em multicylinder} on a (half-)translation surface $(X,\omega)$ is a collection of cylinders whose core curves are all disjoint and non-isotopic (for some, hence any, choice of nonsingular core curves).
A multicylinder is a {\em parallel multicylinder} if each of the core curves of its constituent cylinders has the same slope.
\end{definition}

More generally, every (half-)translation surface has a circle's worth of (singular) {\em directional foliations} by parallel lines of the same slope.
The directional foliations of translation surfaces are always orientable, while those for half-translation surfaces are not necessarily so.
If the directional foliation of a (half-)translation surface $(X,\omega)$ is a union of parallel multicylinders glued along parallel saddle connections, then $(X,\omega)$ is said to be {\em periodic} in that direction.

If $\alpha$ is a parallel multicylinder and $b$ is a cylinder on a translation surface, then the intersection pattern is constrained.

\begin{lemma}
\label{1 parallel implies coherently orientable}
Given a parallel multicylinder $\alpha$ and arbitrary multicylinder $\beta$ on a translation surface, their core multicurves are coherently orientable.
\end{lemma}

In particular, any two cylinders on a translation surface must have coherently orientable core curves.

\begin{proof}
Let $\alpha =a_1 \cup  \dots \cup a_n$ and $\beta =b_1 \cup \dots \cup b_m$.
Rotating the surface as necessary, we may assume that $\alpha$ is horizontal. Orient the core curves of $\alpha$ to point in the $+x$ direction.

The core curve of each $b_i$ has constant slope since it is realized as a cylinder. 
If some $b_i$ does not intersect $\alpha$, then $\geomintersect(\alpha, b_i) = \abs{\algintersect(\alpha, b_i)} = 0$ and its orientation does not matter.
If $b_i$ intersects $\alpha$, then it cannot be horizontal, and so we assign an orientation to $b_i$ so that its geodesic representatives points with angle in $(0, \pi)$.
Doing this for all $b_i$, we have thus oriented $\beta$ so that all intersections with $\alpha$ are positive, making the pair of oriented multicurves coherent.
\end{proof}

\subsection{Realizing topological multicurves as cylinders} \label{subsection_realization}
Throughout the rest of this paper, we will be interested in when pairs of multicurves on a topological surface $S$ of genus $g$ can be realized as a pair of multicylinders on some (half-)translation surface $(X,\omega)$.
However, there is no canonical way of identifying the curves on $S$ with those on $(X,\omega)$.
In order to make this identification, we must fix a marking.

A \emph{marking} of a (half-)translation surface $(X,\omega)$ is a choice of homeomorphism $\marking$ from $S$ to $(X,\omega)$, considered up to isotopy.
\footnote{
If the translation surface has $n$ cone points, one could also consider a marking $\varphi': S_{g,n} \to X$ that records the location of the cone points. With this definition of marking, we cannot isotope curves over cone points because we treat cone points as punctures.
This gives rise to a different notion of realizability as a multicylinder which depends heavily on the position of curves vis-{\`a}-vis cone points. 
See also Remark \ref{rmk:graftingandmarking}.
}
Markings allow us to compare curves on different flat surfaces. Usually one denotes a \emph{marked translation surface} by the triple $(X, \omega, \marking)$. However in this paper, any translation surface $(X,\omega)$ is assumed to be marked, so for simplicity of notation we will usually omit $\marking$ from the triple.
The set of all marked half-translation surfaces is naturally identified with the cotangent bundle of the Teichm{\"u}ller space of $S$, and the set of all marked translation surfaces $(X,\omega,\varphi)$ corresponds to a subbundle of the cotangent bundle.


\begin{definition}\label{def:realize_nonor}
A multicurve $\gamma = c_1 \cup...\cup c_n$ on $S$ {\em can be realized as a multicylinder on a (half-)translation surface} if there exists some marked (half-)translation surface $(X, \omega, \marking)$ so that the curves $\marking(c_i)$ are all isotopic to the core curves of a multicylinder on $(X,\omega)$.
Similarly, we say that $\gamma$ {\em can be realized as a parallel multicylinder} if $\marking(c_i)$ can be isotoped to the core curves of a parallel multicylinder.
\end{definition}

When a multicurve is oriented, we can also introduce a more restrictive criterion for cylindricity that forces the core curves to all point the same way (as opposed to just having the same slope). 

\begin{definition}\label{def:realize_or}
An oriented multicurve $\vgamma= \vc_1 \cup ... \cup \vc_n$ {\em can be realized as a directional multicylinder} on a translation surface if there exists some marked translation surface $(X, \marking)$ and some direction $\theta \in [0, 2\pi)$ so that the curves of $\marking(\vc_i)$ are all isotopic (respecting orientations) to closed nonsingular curves of the directional foliation pointing in direction $\theta$.
\end{definition}

We remark that this definition only makes sense for translation surfaces, as the directional foliations of half-translation surfaces are not always orientable.

Equivalently, Definition \ref{def:realize_or} can also be stated in terms of the periods of the curves of $\vgamma$.
Recall that the {\em period} of any oriented path $p:[0,1] \to X$ on a translation surface is the integral of the differential defining $(X,\omega)$ over the path $p$.
An oriented multicurve $\vgamma$ is then realized as a directional multicylinder on $(X,\omega)$ if and only if it is realized as a parallel multicylinder and the periods of the constituent $\vc_i \subset \vgamma$ are all {\em positive} real multiples of each other.

The property of realizing a given curve as a cylinder is open, as proven by the following Lemma:

\begin{lemma}\label{lem:cylinderspersist}
If a curve $c$ is realized as a cylinder on some (half-)translation surface $(X,\omega)$, then there exists an open neighborhood $N$ of $(X,\omega)$ in 
the ambient stratum of marked (half-)translation surfaces
such that $c$ is realized as a cylinder on every surface $(X', \omega') \in N$.
\end{lemma}

\begin{proof}
If $c$ is realized as a cylinder on $(X,\omega)$, then it has a nonsingular geodesic representative of constant slope.
There is some minimum distance, $d$, between the geodesic and any of the cone points in $\Sigma$.
Varying slightly in period coordinates (equivalently, varying the edges of the defining polygons), $(X,\omega)$ can then be deformed into a new surface $X'$ such that each point of $\Sigma$ remains at least $d/2>0$ away from $c$, and on this surface, $c$ is a geodesic line of constant slope. Thus, $c$ is still realizable as a cylinder on $(X', \omega')$. 
\end{proof}

The same argument shows that realizing a multicurve as a multicylinder is also an open condition, but realizing it as a parallel or directional multicylinder is not. We say more to this effect in Lemma \ref{lemma:tv-directional-equivalence}.

\subsection{The Thurston--Veech construction}
In this subsection, we recall a method of finding a marked translation surface which will realize certain pairs of multicurves $\valpha$ and $\vbeta$ as directional multicylinders.

A pair of multicurves $(\alpha, \beta)$ \textit{fills} $S$ if $S \setminus (\alpha \cup \beta)$ is a disjoint union of open disks. It is a standard result that a pair of multicurves $(\alpha, \beta)$ fills $S$ if and only if for every curve $c \subset S$ that is not null-homotopic, $\geomintersect(c, \alpha) + \geomintersect(c, \beta) > 0$.
In the case that $(\alpha,\beta)$ is a filling pair of multicurves, Thurston and Veech independently identified a now standard construction which realizes $\alpha$ and $\beta$ as cylinders on a half-translation surface.
Heuristically, one can think of this procedure as thickening $\alpha$ and $\beta$ into cylinders while contracting the remaining disks into cone points \cite{NSnotes}. A more formal statement of the construction is given below. 

\begin{construction}[Thurston--Veech Construction] 
\label{construction:tv}
Let $(\alpha, \beta)$ be a filling pair of non-empty multicurves on $S$ and think of $\alpha \cup \beta$ as an embedded graph, with vertices corresponding to points of intersection and edges corresponding to strands of $\alpha$ and $\beta$ running between intersection points.
Then the dual graph of $\alpha \cup \beta$ partitions $S$ into squares.
Considering these to be flat unit squares, we obtain a \emph{square-tiled surface} on which the thickened curves of $\alpha$ comprise the horizontal cylinders and the thickened curves of $\beta$ comprise the vertical cylinders. See Figure \ref{fig:TVexample}. Because $\alpha$ is always horizontal and $\beta$ is always vertical, edges of squares are only identified by translations and rotations by $\pi$. Thus, the resulting surface $\TV(\alpha,\beta)$ is a half-translation surface.
\end{construction}

We observe that by construction, any surface $\TV(\alpha,\beta)$ is periodic in both the horizontal and vertical directions and is tiled by squares. 
Such surfaces are sometimes called ``origamis'' in the literature.

\begin{figure}
    \centering
    \def\svgwidth{250 pt}
    \includegraphics[width=.6\textwidth]{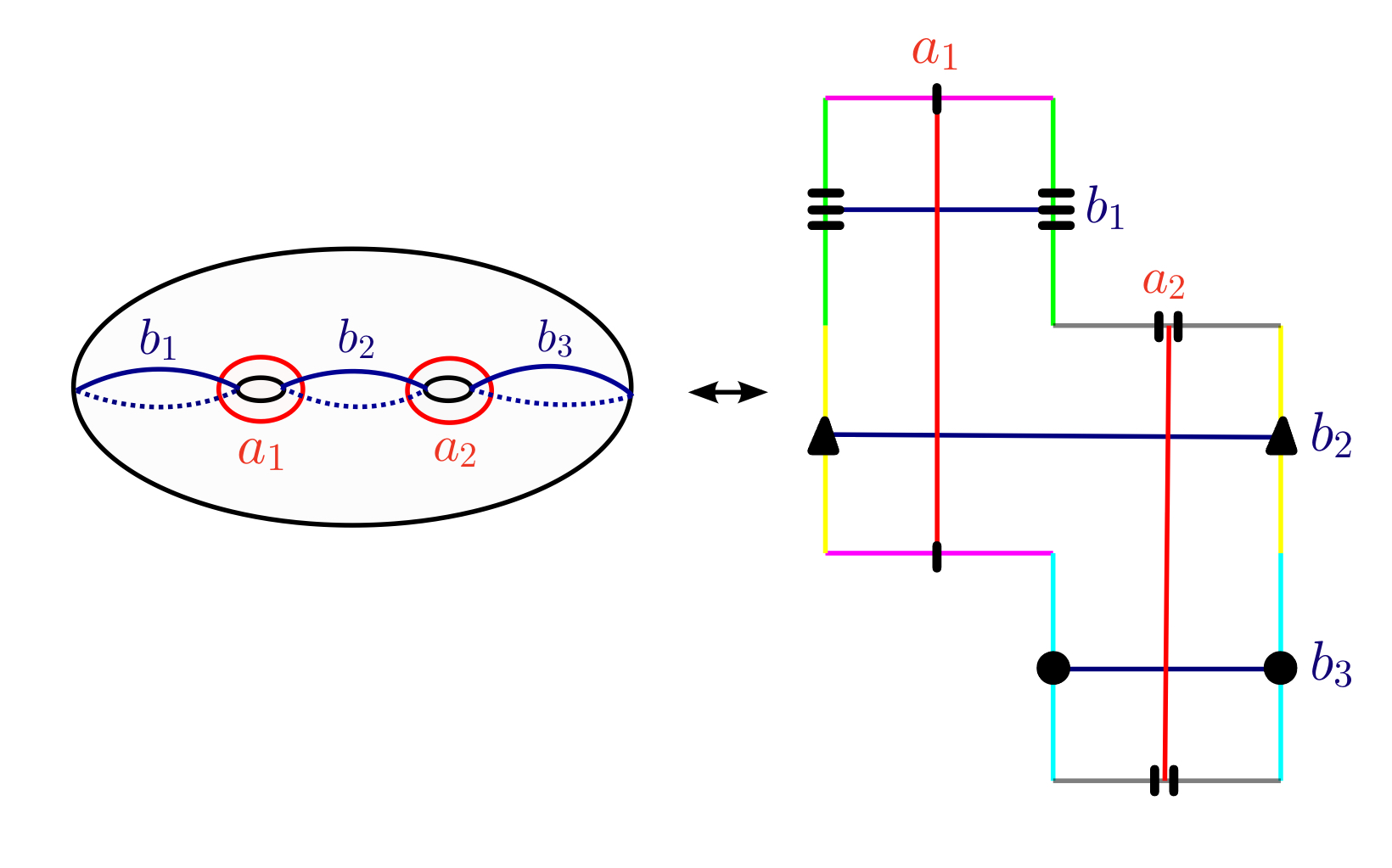}
    \caption{An example of the Thurston--Veech construction applied to the pair of multi-curves $\alpha = a_1 \cup a_2$ and $\beta = b_1 \cup b_2 \cup b_3$.}
    \label{fig:TVexample}
\end{figure}

In this paper, we restrict our attention to translation surfaces because half-translations are too flexible, as demonstrated in Lemma \ref{lem:STScyls} below.

Recall that a {\em pants decomposition} of $S$ is a multicurve $\alpha$ so that $S \setminus \alpha$ is a union of three-holed spheres.
An \textit{extension} of a multicurve $\alpha$ on $S$ is a multicurve $\alpha'$ on $S$ which contains $\alpha$ and preserves the orientations, if any, of the curves of $\alpha$. A pair of multicurves $(\alpha, \beta)$ on $S$ is \textit{extended} to a pair of multicurves $(\alpha', \beta')$ such that $\alpha'$ is an extension of $\alpha$ and $\beta'$ is an extension of $\beta$.

\begin{lemma}\label{lem:STScyls}
Every pair of multicurves $\alpha$ and $\beta$ with no curves in common is realized as a pair of parallel multicylinders on some half-translation surface.
\end{lemma}
\begin{proof}
Choose two pants decompositions $P_\alpha$ and $P_\beta$ extending $\alpha$ and $\beta$, respectively, that have no curves in common. This can be accomplished, for example, by extending both to arbitrary pants decompositions and then applying an ``elementary move'' to any curves they have in common. See \cite{HT}.

We now claim that $P_\alpha$ and $P_\beta$ jointly fill the surface. Indeed, since pairs of pants have no non-peripheral simple curves, if a curve $c$ does not meet $P_\alpha$, then it must be a curve of $P_\alpha$.
Since $P_\alpha$ and $P_\beta$ share no curves, this means that $c$ must meet $P_\beta$, and thus $(P_\alpha, P_\beta)$ fill the surface.
Applying Construction \ref{construction:tv} therefore produces a half-translation surface on which $P_\alpha$ and $P_\beta$, hence $\alpha$ and $\beta$, are both realized as parallel multicylinders.
\end{proof}

Adding the requirement that a pair of multicurves is coherently orientable guarantees Construction \ref{construction:tv} will yield a translation surface. 

\begin{lemma}
\label{lemma:coherence-tv}
Given a filling pair of multicurves $\alpha$ and $\beta$ on $S$, the surface $\TV(\alpha,\beta)$ is a translation surface if and only if $\alpha$ and $\beta$ are coherently orientable.
\end{lemma}

This result has already been shown in \cite{CHANG2021107730, jeffreys2021single}, but we include a short proof for completeness.

\begin{proof}
The forward direction follows from Lemma \ref{1 parallel implies coherently orientable}. 

In the backward direction, suppose that $\valpha$ and $\vbeta$ form a coherent filling pair of multicurves. Then $\TV(\valpha, \vbeta)$ can be represented by a collection of squares, each containing a single intersection of $\valpha$ and $\vbeta$. Rotating the squares so that $\valpha$ runs left-to-right, the coherence condition ensures that, without loss of generality, $\vbeta$ runs bottom-to-top. Therefore, top edges of squares are identified with bottom edges of squares, and right edges are identified with left edges. Thus, the resulting square-tiled surface is a translation surface.
\end{proof}

In particular, this Lemma tells us that if a coherent pair of oriented multicurves $\valpha$ and $\vbeta$ can be extended to a coherent filling pair, then they can be realized as the horizontal and vertical cylinders of some square-tiled surface.
It turns out the converse is also true:

\begin{lemma}
\label{lemma:tv-directional-equivalence}
Let $(\valpha, \vbeta)$ be a coherent pair of oriented multicurves on $S$. Then there exists a translation surface $(X, \omega)$ on which $(\valpha, \vbeta)$ are realized as directional multicylinders if and only if $(\valpha, \vbeta)$ can be extended to a coherent filling pair.
\end{lemma}
\begin{proof}
Suppose that $\valpha = \vec{a}_1 \cup \ldots \cup \vec{a}_m$ and $\vbeta = \vec{b}_1 \cup \ldots \cup \vec{b}_n$ are directional multicylinders on a translation surface $(X, \omega)$.
Postcomposing with an element of $\GL_2(\RR)$ as necessary, we may assume that the geodesic representatives of each curve of $\valpha$ points in the positive $x$-direction and each curve of $\vbeta$ in the positive $y$-direction.
That is, we have that
\begin{equation}\label{eq:reandim}
\hol(\vec{a_i}) \in \RR_{>0} \text{ and }
\hol(\vec{b_j}) \in i\RR_{>0}
\end{equation}
for all $i =1, \ldots, m$ and $j = 1, \ldots, n$.

Choose a local period coordinate chart for the ambient stratum around $(X, \omega)$; then the equations
\[\Im ( \hol(\vec{a_i}) ) = 0 \text{ and }
\Re( \hol(\vec{b_j}) ) = 0\]
cut out a (nonzero) $\RR$-linear subspace $V$ of $H^1(S, \text{Zeros}(\omega); \CC)$.
The positivity conditions of \eqref{eq:reandim} further specify an intersection of open half-spaces inside of $V$, which is non-empty because it contains $(X, \omega)$. Let $U \subset V$ denote this intersection; note that $U$ is a relatively open set inside of $V$.

Since $V$ is cut out by integral-linear equations, rational points $H^1(S, \text{Zeros}(\omega); \QQ \oplus i \QQ)$ are dense in $V$.
Since cylinders persist under small deformations (Lemma \ref{lem:cylinderspersist}), this implies that there is some
\[(X', \omega') \in H^1(S, \text{Zeros}(\omega); \QQ \oplus i \QQ) \cap U\]
on which $\valpha$ and $\vbeta$ remain cylinders.
In particular, $(X',\omega')$ is a square-tiled surface on which $\valpha$ and $\vbeta$ are horizontal and vertical cylinders.
The entire horizontal and vertical multicurves of $(X',\omega')$, oriented in the positive $x$- and $y$-directions respectively, therefore constitute a (coherent) filling pair extending $\valpha$ and $\vbeta$.

The backwards direction is just Lemma \ref{lemma:coherence-tv}.
\end{proof}

\section{Coherent, filling pairs of multicurves}\label{sec:coherent-filling-pairs-of-multicurves}
Given a pair of oriented multicurves $(\valpha,\vbeta)$ on $S$, the strongest sense in which we could realize $\valpha$ and $\vbeta$ as multicylinders on a translation surface is as a directional multicylinders.
The goal of this section is to prove a refinement of Theorem \ref{mainthm:parallel}, characterizing exactly when this is possible.

Let $S$ be a surface and $\vc$ an oriented simple closed curve on it. The surface $S \setminus \vc$ is obtained by removing a small annular neighborhood of $\vc$; we denote the two oriented boundary components of that neighborhood as well as the corresponding boundary components of $S \setminus \vc$ by $\vc_L$ and $\vc_R$, depending on whether the surface is on the left- or right-hand side of the curve. Throughout this section, we assume that none of the constituent curves in a pair of multicurves are isotopic.

\begin{theorem}
\label{theorem:exttv_oriented}
Let $(\valpha,\vbeta)$ be a pair of coherent oriented multicurves such that no two curves of $\valpha \cup \vbeta$ are isotopic.
Then $(\valpha,\vbeta)$ can be simultaneously realized as directional multicylinders on a translation surface if and only if the following holds:
\begin{equation}
    \tag{$\star$}\label{noobstruction}
\parbox{\dimexpr\linewidth-4em}{
\strut
For every multicurve $\vgamma \subseteq \valpha \cup \vbeta$ and every complementary subsurface $W$ of $S \setminus \vgamma$, partition
\[\partial W = A_L \sqcup A_R \sqcup B_L \sqcup B_R\]
where $A_L$ denotes the set of boundary components arising from the left-hand sides of curves of $\valpha$, and so on. Then $A_L \neq \emptyset$ if and only if $A_R \neq \emptyset$, and $B_L \neq \emptyset$ if and only if $B_R \neq \emptyset$.
\strut
} 
\end{equation}
\end{theorem}

In particular, we note that if $(\valpha, \vbeta)$ satisfies \eqref{noobstruction} then none of the curves of $\valpha$ or $\vbeta$ are separating.

The proof of this theorem is spread throughout the section. In Section \ref{subsec:obstructexttv}, we give both geometric and topological proofs that condition \eqref{noobstruction} is necessary.
In light of the equivalence of Lemma \ref{lemma:tv-directional-equivalence}, it is enough to prove that a coherent pair of multicurves $(\valpha, \vbeta)$ satisfying the hypotheses of Theorem \ref{theorem:exttv_oriented} can be extended to a coherent, filling pair. 
Our plan is to use an inductive ``connected sum'' construction to extend the pair $(\valpha, \vbeta)$, building new curves out of, and informed by, existing ones, to decrease the complexity of $S \setminus (\alpha \cup \beta)$ at each step.
However, it turns out curves of $\valpha$ which do not meet any of the curves of $\vbeta$ (or vice versa), called \textit{singletons}, complicate this strategy.
As such, in Section \ref{subsec:nosingletons} we extend $(\valpha, \vbeta)$ to a coherent pair $(\valpha', \vbeta')$ in which there are no singletons, that is, so that each curve of $\valpha'$ intersects some curve of $\vbeta'$, and vice versa.
We then use connected sums in Section \ref{subsec:connectsums} to add curves to $\vbeta'$ that are all coherent with $\valpha'$, thereby completing the proof of Theorem \ref{theorem:exttv_oriented}.
This section also contains an explanation of how Theorem \ref{theorem:exttv_oriented} implies Theorem \ref{mainthm:parallel}.

For the rest of the paper, let us assume (without loss of generality) that $i(\valpha, \vbeta) \ge 0$, so that the curves of $\vbeta$ cross the curves of $\valpha$ from right to left.

\subsection{Necessity}\label{subsec:obstructexttv}

The first step to proving Theorem \ref{theorem:exttv_oriented} is to show that condition \eqref{noobstruction} is necessary.

\begin{lemma}
\label{lemma:obstruction}
Let $(\valpha, \vbeta)$ be a pair of multicurves on $S$ with $\geomintersect(\valpha, \vbeta) > 0$.
Suppose that there exists a multicurve $\vgamma \subset \valpha \cup \vbeta$ and component $W$ of $S \setminus \vgamma$ with $A_L \neq \emptyset$ but $A_R = \emptyset$. Then $(\valpha, \vbeta)$ are not realizable as a pair of directional multicylinders on any translation surface.
\end{lemma}

\begin{proof}
Suppose that $\valpha = \bigcup_{i} \va_i$ and $\vbeta = \bigcup_{j} \vb_j$ were realized as a pair of directional
multicylinders on the marked translation surface $(X, \omega, \marking)$.
Because $\valpha$ is realized as a directional multicylinder, all of the period vectors $\hol(\va_i)$ live in the same ray $\hol(\va_1) \cdot \RR_{>0} \subset \CC$. Similarly, there is a ray containing all of the period vectors $\hol(\vb_j)$. Because $\valpha$ and $\vbeta$ intersect, the core curves of the corresponding cylinders are not parallel, and so the corresponding rays are not parallel.
  
By Stokes's theorem, we have 
\[ \sum_{\va_L \in A_L} \hol(\va_L) + \sum_{\vb_L \in B_L} \hol(\vb_L) - \sum_{\vb_R \in B_R}\hol(\vb_R)
  = \int_{\partial W} \omega = \int_W d\omega = 0, \]
because holomorphic forms on a Riemann surface are always closed. By linear independence, $\hol(\vec{a}_L) = 0$ for all $\va_L \in A_L$.
Since cylinders never have 0 circumference, this is a contradiction.
\end{proof}

We note that the same proof holds if one allows $\vbeta$ to be empty, implying that no separating curve is realized as a cylinder on any translation surface.

While the following result is not strictly necessary for the proof of Theorem \ref{theorem:exttv_oriented}, we observe that Lemma \ref{lemma:tv-directional-equivalence} gives us another way to understand the obstruction from Lemma \ref{lemma:obstruction} in more topological terms.

\begin{corollary}\label{cor:top version of obstruction}
Let $(\alpha, \beta)$ be a pair of unoriented multicurves on $S$. If some curve $a \subset \alpha$ separates $S \setminus \beta$, then $\alpha$ and $\beta$ cannot be simultaneously realized as parallel multicylinders on a translation surface.
\end{corollary}

\begin{proof}
Suppose that $\alpha$ and $\beta$ were both realized as parallel multicylinders on some translation surface. By Lemma \ref{lemma:tv-directional-equivalence}, this implies that we can extend $(\alpha, \beta)$ to a coherently orientable filling pair $(\alpha', \beta')$.

Now if some curve $a \subset \alpha$ separates $S \setminus \beta$, there exists some subsurface $W$ one of whose boundary components corresponds to $a \subset \alpha$ and the rest of which correspond to curves of $\beta$.
Since $\beta'$ fills with $\alpha' \supset \alpha$, this implies there is some curve $b \subset \beta'$ with $\geomintersect(a, b) > 0$.
But now since $b$ is disjoint from the other curves of $\beta'$, it cannot intersect any of the other boundary curves of $W$, so it must cross $a$ at least twice: one time entering $W$, and one time escaping $W$.
This is a contradiction with coherence, and so we see that $(\alpha, \beta)$ cannot be realized as a pair of parallel multicylinders.
\end{proof}

\subsection{Removing singletons}\label{subsec:nosingletons}
We now begin to build towards a proof that the phenomenon described in Lemma \ref{lemma:obstruction} is the only obstruction to realizing a coherent pair of multicurves as a pair of directional multicylinders.
The first step is to show that any coherent pair of oriented multicurves $(\valpha, \vbeta)$ satisfying \eqref{noobstruction} can be extended to a coherent pair $(\valpha', \vbeta')$ where every curve of $\valpha'$ intersects a curve of $\vbeta'$, and vice versa.

\begin{definition}
The {\em singleton set } $\smash{\singletons_{\vbeta}} (\valpha)$ of $\valpha$ with respect to $\vbeta$ is the set of curves of $\valpha$ which do not intersect any curve of $\vbeta$ when realized in minimal position.
\end{definition}

Given $\va \in \smash{\singletons_{\vbeta} (\valpha)}$, we wish to construct a curve $\vb$ intersecting $\va$ such that $(\valpha, \vbeta \cup \vb)$ remains coherent.
Our strategy is to concatenate arcs which are all coherent with $\valpha$ but disjoint from $\vb$.
First, we need to show that there is a sufficient supply of such arcs (Lemma \ref{lem:arcs_on_subsurfs} below).

Let us first recall the notion of a connected sum of curves.

\begin{definition}\label{def:connsum}
Let $a$ be a simple curve or arc on a surface and let $b$ be a disjoint simple closed curve. Let $\varepsilon$ be an arc connecting $a$ to $b$, disjoint from $a$ and $b$ except at its endpoints. The {\em connected sum} $a +_\varepsilon b$ is the curve (or arc) obtained by taking the boundary of a tubular neighborhood of $a \cup \varepsilon \cup b$.
\end{definition}

If $\va$ and $\vb$ are oriented and $\varepsilon$ runs from the left-hand side of $\va$ to the left-hand side of $\vb$, then moreover one can orient $\va +_\varepsilon 
\vb$ so that it runs parallel to $\va$ and $\vb$ away from $\varepsilon$. See Figure \ref{fig:connected-sum}.

\begin{lemma}\label{lem:arcs_on_subsurfs}
Let $(\valpha, \vbeta)$ be a coherent pair of oriented multicurves and let $W$ be a component of $S \setminus (\vbeta \cup \singletons_{\vbeta}(\valpha))$. 
Let $A_L$ ($A_R$) denote the set of boundary components of $W$ arising from the left- (right-)hand sides of curves of $\singletons_{\vbeta} (\valpha)$.
Then for every $\va_L \in A_L$ and $\va_R \in A_R$, there is an oriented arc on $W$ traveling from $\va_L$ to $\va_R$ that crosses $\valpha|_W$ from right to left.
\end{lemma}

\begin{figure}
    \centering
    \def\svgwidth{450 pt}
    \includegraphics[width=.8\textwidth]{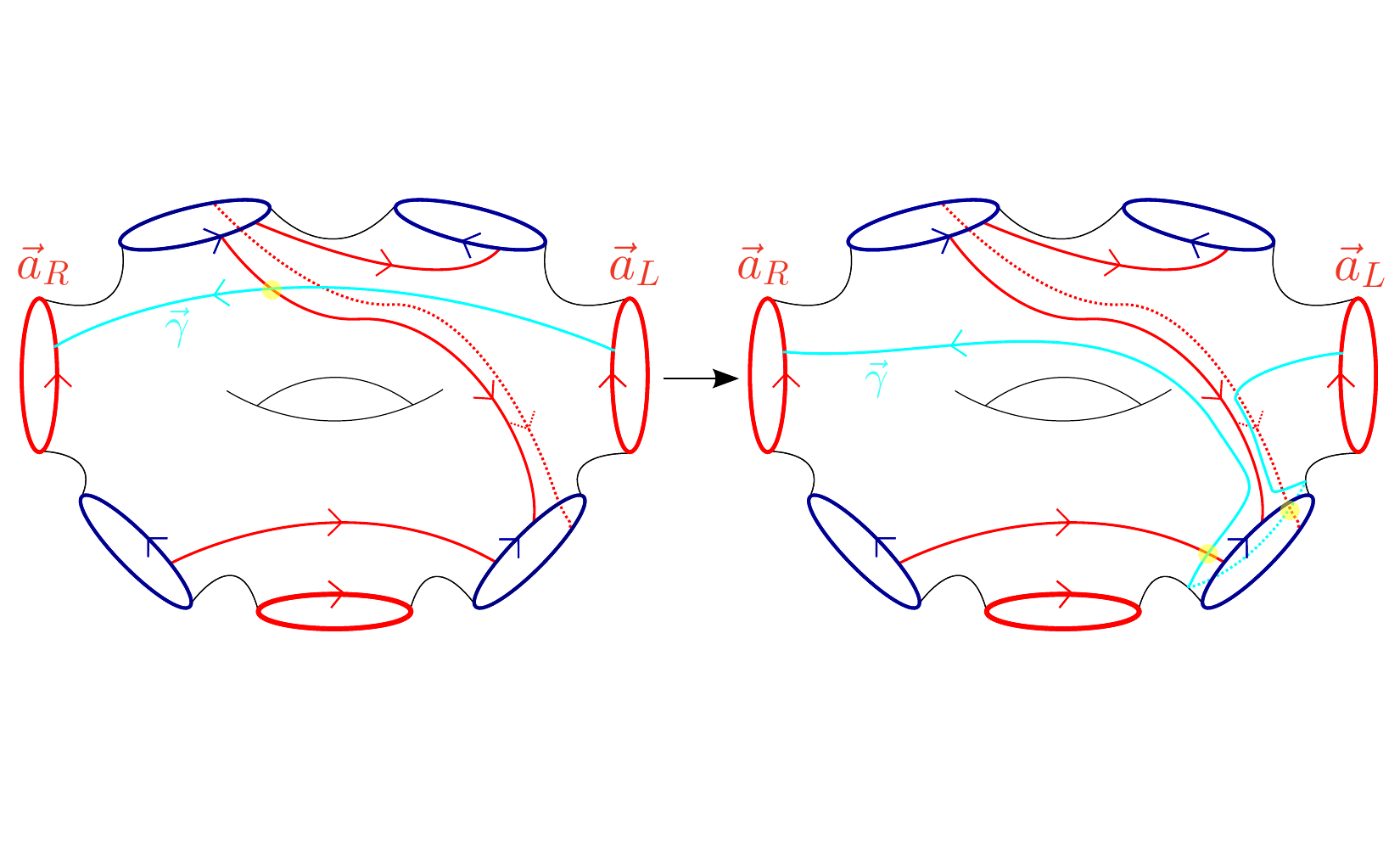}
    \caption{Replacing a non-coherent intersection with coherent ones}
    \label{fig:connected-sum}
\end{figure}

\begin{proof}
Given any $\va_L \in A_L$ and any $\va_R \in A_R$, let $\vgamma$ be an arbitrary oriented arc on $W$ connecting initial boundary component $\va_L$ to terminal boundary component $\va_R$.
It is possible that $\vgamma$ intersects arcs of $\valpha|_W$ from left to right (when realized in minimal position).
Our strategy is to surger $\vgamma$ to remove these intersection points.

So suppose that $\vgamma$ crosses an arc $\mathfrak{a}$ of $\valpha|_W$ from left to right. 
Order all of the points $p_1, \ldots, p_N$ of $\mathfrak{a} \cap \gamma$ as seen by the oriented arc $\vgamma$.
We may now build a new arc $\vgamma'$ obtained by following $\vgamma$ from $\va_L$ to $p_1$, then traveling along $\mathfrak{a}$ from $p_1$ to $p_N$, then following $\vgamma$ from $p_N$ to $\va_R$.
This arc has the same endpoints as $\vgamma$ but intersects $\mathfrak{a}$ at most once; if it is disjoint or crosses $\mathfrak{a}$ from right to left, then we are done.

Otherwise, $\vgamma'$ crosses $\mathfrak{a}$ from left to right exactly once. In this case, let $\varepsilon$ denote the subarc of $\mathfrak{a}$ running from $\mathfrak{a} \cap \gamma'$ to the boundary curve $\vb$ containing the terminal endpoint of $\mathfrak{a}$.
Then we may take the connect sum of $\vgamma'$ and $\vb$ along $\varepsilon$, adopting the orientations of $\vgamma'$ and of $\vb$. This procedure has the effect of wrapping $\vgamma'$ around $\vb$, thereby removing the left-to-right intersection of $\vgamma'$ with $\mathfrak{a}$. While the connected sum may introduce new intersections with $\valpha|_W$, the coherence of $\valpha$ and $\vbeta$ guarantees that $\vgamma'$ crosses from right to left at these new intersections.

In either case, we have built an arc with the same endpoints as $\vgamma$ with at least one fewer left-to-right intersection with $\valpha|_W$. Repeating this procedure, we are left with an arc on $W$ from $\va_L$ to $\va_R$ that crosses arcs of $\valpha|_W$ only from right to left.
\end{proof}

We now show that we can piece together these arcs coherently.

\begin{lemma}\label{lem:addcurve_removesingleton}
Let $(\valpha, \vbeta)$ be a coherent pair of oriented multicurves satisfying \eqref{noobstruction}. Given a singleton $\va \in \smash{\singletons_{\vbeta}}(\valpha)$, there is an oriented curve $\vb$ that meets $\va$, is disjoint from $\vbeta$, and whose intersections $\valpha$ are all positive.
\end{lemma}

\begin{proof}
If $\va$ is non-separating on $S \setminus \beta$, then a standard change-of-coordinates argument (see, e.g., \cite[\S 1.3.3]{FarbMarg}) implies that there must be an unoriented curve $b \subset S \setminus \beta$ meeting $a$ exactly once, and we can choose its orientation to satisfy coherence.

Otherwise, let $W_1, \ldots, W_n$ denote the complementary subsurfaces of
$S \setminus (\vbeta \cup \singletons_{\vbeta}(\valpha))$.
Build a directed graph $G$ whose vertices are the subsurfaces $W_i$ and so that there is an edge from $W_i$ to $W_j$ if $W_i$ is on the left and $W_j$ is on the right of some curve $\va' \in \singletons_{\vbeta}(\valpha)$.
Connected components of this graph correspond to components of $S \setminus \beta$.
See Figure \ref{fig:digraph}.

\begin{figure}
    \centering
    \def\svgwidth{450 pt}
    \includegraphics[width=.8\textwidth]{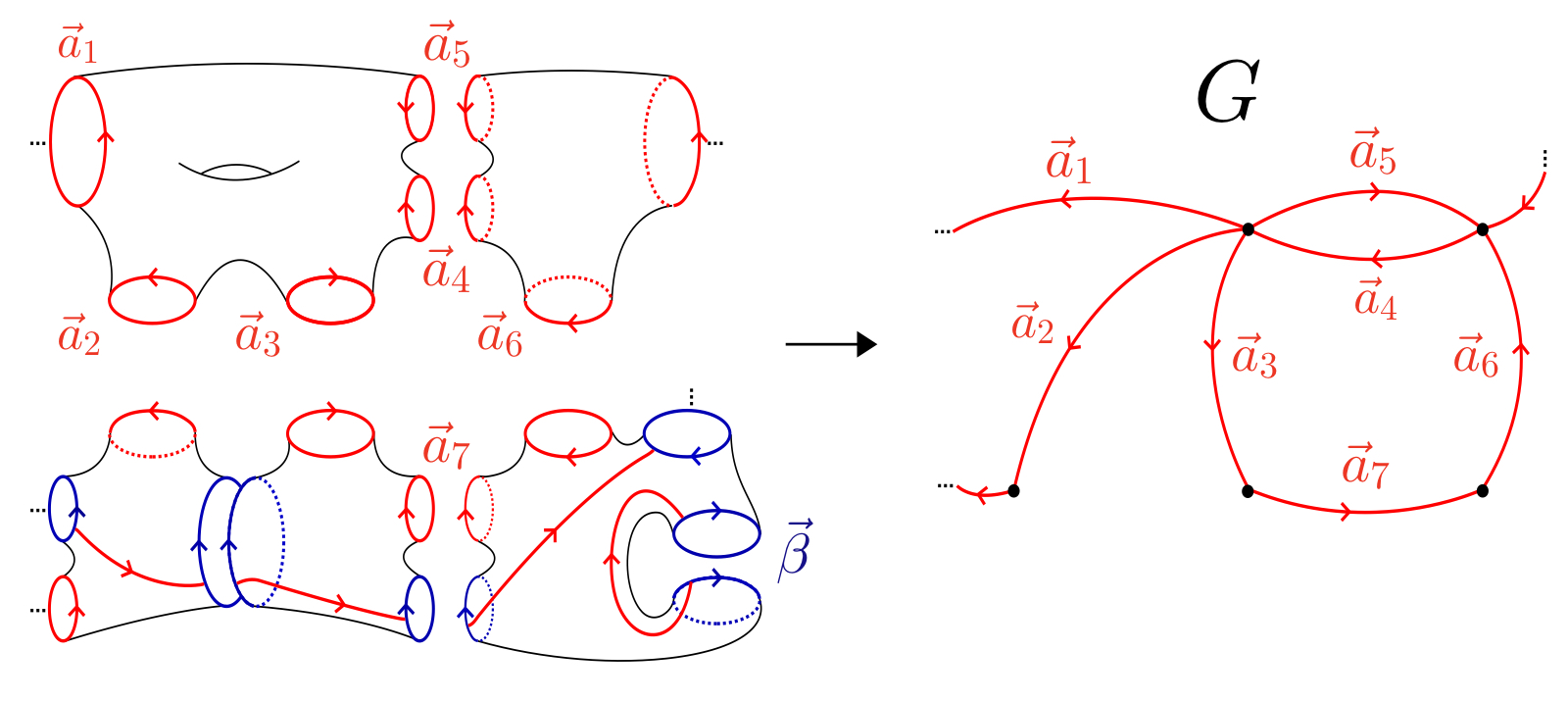}
    \caption{Building a graph out of components of $S \setminus (\vbeta \cup \singletons_{\vbeta}(\valpha))$.}
    \label{fig:digraph}
\end{figure}

\begin{claim}
Each connected component $G_0$ of $G$ is strongly connected, i.e., there is a directed path between any two vertices of $G_0$.
\end{claim}

\begin{proof}
We prove an equivalent definition of strong connectivity: for any edge cut of $G_0$ (i.e., any partition of its vertices into two sets), there are directed edges traveling from one set of the partition to the other and vice versa.

Consider an arbitrary partition
\[V(G_0) = V_1 \sqcup V_2\]
of the vertices of $G_0$.
Let $E_{12}$ denote set of oriented edges running from $V_1$ to $V_2$ and $E_{21}$ the set of oriented edges from $V_2$ to $V_1$. Let 
$\valpha_E \subset \smash{\singletons_{\vbeta}}(\valpha)$
be the oriented multicurve corresponding to $E = E_{12} \cup E_{21}$.

Now choose any component $W_E$ of $S \setminus (\vbeta \cup \valpha_E)$; note that $W_E$ is necessarily built by gluing subsurfaces corresponding to vertices of either $V_1$ or $V_2$, but not both. 
The components of $\partial W_E$ are oriented depending on whether they correspond to edges from $E_{12}$ or $E_{21}$.
For example, if $W_E$ consists of a union of subsurfaces corresponding to vertices of $V_1$, then the curves of $\partial W_E$ arising from $E_{12}$ are oriented with $W_E$ on their left.

Now because $(\valpha, \vbeta)$ satisfies \eqref{noobstruction}, we therefore know that both $E_{12}$ and $E_{21}$ must be nonempty, hence there are edges running from $V_1$ to $V_2$ as well as edges running from $V_2$ to $V_1$.
Since our partition was arbitrary, this allows us to deduce strong connectivity of $G_0$.
\end{proof}

Consider now the edge of $G$ corresponding to our chosen singleton $\va$ and let $W_i$ and $W_t$ denote its initial and terminal vertices. Because each component of $G$ is strongly connected, there is a directed (simple) path from $W_t$ to $W_i$.
This path corresponds to a sequence of subsurfaces
\[W_t = W_1, W_2, \ldots, W_N = W_i\]
where $W_j$ lies to the left and $W_{j+1}$ to the right of some curve $\va_j \in \singletons_{\vbeta}(\valpha)$.

On each $W_j$, use Lemma \ref{lem:arcs_on_subsurfs} to choose an oriented arc connecting $\va_j$ to $\va_{j+1}$ and crossing $\valpha$ from right to left (where indices are interpreted mod $N$). 
We may then concatenate these arcs, possibly with a partial twist around the curves of $\va_j$ to ensure that endpoints match up.
This yields a curve which crosses each $\va_j$ (in particular, crosses $\va$), is disjoint from $\vbeta$, and each of its intersections with $\valpha$ is positive.
\end{proof}

Iterating this Lemma lets us remove the singletons of $\valpha$ by adding curves to $\vbeta$, none of which are themselves singletons. Then, we can swap the roles of $\valpha$ and $\vbeta$ to remove the singletons of $\vbeta$ by adding curves to $\valpha$.
For later use, we record this as the following:

\begin{proposition}\label{prop:extendnosingletons}
Let $(\valpha, \vbeta)$ be a coherent pair of oriented multicurves satisfying \eqref{noobstruction}.
Then there are oriented multicurves $\valpha' \supset \valpha$ and $\vbeta' \supset \vbeta$ such that $(\valpha', \vbeta')$ is coherent, satisfies \eqref{noobstruction}, and has no singletons. That is, every curve of $\valpha'$ meets some curve of $\vbeta'$ and vice versa.

Moreover, if $\singletons_{\vbeta}(\valpha) = \emptyset$ then $\vbeta'$ can be taken to equal $\vbeta$ and if 
$\singletons_{\valpha}(\vbeta) = \emptyset$ then $\valpha'$ can be taken to equal $\valpha$.
\end{proposition}
\begin{proof}
The only thing yet to prove is that if $\vb$ is a curve obtained from Lemma \ref{lem:addcurve_removesingleton} then the pair $(\valpha, \vbeta \cup \vb)$ still satisfies \eqref{noobstruction}. This is also what allows us to iteratively apply Lemma \ref{lem:addcurve_removesingleton}.

So suppose that $\vgamma \subset \valpha \cup \vbeta \cup \{\vb\}$ is a multicurve and $W$ is a complementary subsurface of $S \setminus \vgamma$. Partition
\[\partial W = A_L \sqcup A_R \sqcup B_L \sqcup B_R.\]
If none of the boundary components of $W$ come from $\vb$, then the conclusion of \eqref{noobstruction} follows because $(\valpha, \vbeta)$ satisfies \eqref{noobstruction}.

Otherwise, without loss of generality, we may assume that $\vb \in B_R$.
Now by our choice of $\vb$, there is some curve $\va$ of $\valpha$ crossing $\vb$ from left to right. Since the pair $(\valpha, \vbeta \cup \vb)$ is coherent, and since $\va$ cannot cross either $A_L$ or $A_R$, we see that when $\va$ exits $W$ it must have done so by crossing a curve $\vb'$.
Coherence now implies that $\vb' \in B_L$, and in particular $B_L$ is nonempty.
\end{proof}

\subsection{Extension to filling}\label{subsec:connectsums}
We have reduced to the case where neither of our multicurves has any singletons with respect to the other; as a result, each component of $\valpha \cup \vbeta$ contains curves of both $\valpha$ and $\vbeta$.
We may now extend the pair to fill the surface.

\begin{lemma}\label{lem:nosingletons_extendtofill}
Suppose that $(\valpha, \vbeta)$ is a coherent pair so that every curve of $\valpha$ intersects a curve of $\vbeta$, and vice versa.
Then there exists a multicurve $\vbeta' \supset \vbeta$ so that $(\valpha, \vbeta')$ remains coherent and fills $S$.
\end{lemma}
\begin{proof}
Let $Y$ denote the (possibly disconnected) subsurface obtained by removing $\valpha \cup \vbeta$.
Our proof proceeds by iteratively adding curves to reduce the size of $Y$.
The assumption that the curves of $\valpha$ and $\vbeta$ meet implies that each boundary component of $Y$ is a concatenation of segments of both $\valpha$ and $\vbeta$.
See Figure \ref{fig:addcurve nosingletons}.

\begin{figure}[hb]
    \centering
    \includegraphics[scale=.13]{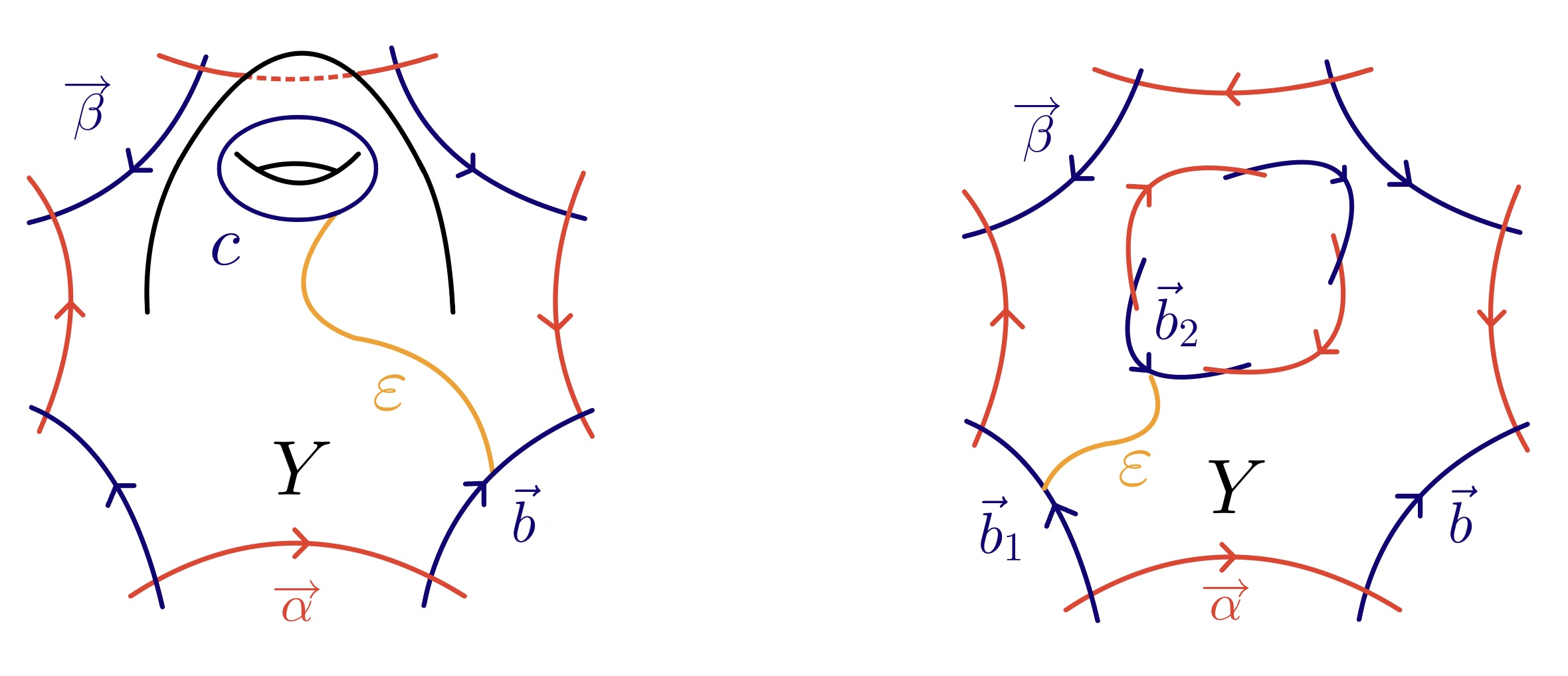}
    \caption{Adding curves to decrease the complexity of the complement of $\alpha \cup \beta$.}
    \label{fig:addcurve nosingletons}
\end{figure}

So as long as $Y$ is not an annulus or pair of pants, there is a non--boundary parallel simple closed curve $c$ on $Y$.
Pick an arbitrary segment of $\vbeta$ that comprises $\partial Y$ corresponding to a curve $\vb$ of $\vbeta$ and pick an arc $\varepsilon$ connecting $c$ to $\vb$ (and otherwise disjoint from $\alpha \cup \beta \cup c$).
We can then form the connect sum $\vb' = \vb +_\varepsilon c$ and orient it so that its orientation agrees with that of $\vb$ as it is running parallel to $\vb$. 
Since $\vb$ crosses $\valpha$ from right to left, so does $\vb'$.
Thus, the pair $(\valpha, \vbeta \cup \vb')$ remains coherent and we see that the Euler characteristic of each component of $\chi(Y \setminus (c \cup \varepsilon))$ is strictly greater than that of $Y$.
This follows because a neighborhood of $c \cup \varepsilon$ together with the relevant component of $\partial Y$ is homeomorphic to a pair of pants.

In the case that $Y$ is an annulus or pair of pants, all of its boundaries are concatenations of segments of $\valpha$ and $\vbeta$. The condition that $\valpha$ and $\vbeta$ are coherent further implies that each boundary component contains some segment of $\vbeta$ that has $Y$ on its right-hand side. Pick curves $\vb_1$ and $\vb_2$ (possibly equal) corresponding to these segments and $\varepsilon$ an arc in $Y$ connecting the segments.
The connect sum $\vb' = \vb_1 +_\varepsilon \vb_2$ is then naturally oriented and is coherent with $\valpha$. In particular, $\vb'$ has positive intersection number with $\valpha$ and so is not nulhomotopic.
The complement of $\vb' \cap Y$ in $Y$ is a disk and a surface homeomorphic to $Y \setminus \varepsilon$, so we see in these cases we can also add a curve to $\vbeta$ and increase the Euler characteristic of each piece of the complement.

Therefore, we may iteratively add curves to $\vbeta$ until every component of $Y$ has Euler characteristic $1$ (equivalently, until every component has no essential arcs), i.e., until every component of $Y$ is a disk.
\end{proof}

We can now put the pieces together to prove our main theorem.

\begin{proof}[Proof of Theorem \ref{theorem:exttv_oriented}]
The necessity of condition \eqref{noobstruction} was proven in Lemma \ref{lemma:obstruction}.

To see that condition \eqref{noobstruction} is sufficient, let $(\valpha, \vbeta)$ be a coherent pair satisfying \eqref{noobstruction}.
By Lemma \ref{lem:addcurve_removesingleton}, we may extend to a coherent pair $(\valpha', \vbeta')$ without any singleton curves.
Applying Lemma \ref{lem:nosingletons_extendtofill}, we can then find a $\vbeta'' \supset \vbeta'$ so that $(\valpha', \vbeta'')$ is a coherent filling pair.
Applying the Thurston--Veech construction (Construction \ref{construction:tv}) to $(\valpha', \vbeta'')$, we obtain a square-tiled surface on which the curves of $\valpha'$ (hence those of $\valpha$) are horizontal cylinders and the curves of $\vbeta''$ (hence those of $\vbeta$) are vertical cylinders.
\end{proof}

Analyzing the steps in our proof, we can similarly give a criterion for when an oriented multicurve can be realized as the entire horizontal foliation of a translation surface.

\begin{corollary}
\label{horizontals}
If $\valpha$ is an oriented multicurve, then there exists a horizontally periodic translation surface $(X, \omega)$ with directional foliation $\valpha$ (oriented in the $+x$-direction)
if and only if for each $\vgamma \subseteq \valpha$ and each complementary subsurface $W$ of $S \setminus \vgamma$, we have that 
\[A_L \neq \emptyset \text{ and } A_R \neq \emptyset\]
where $A_L$ ($A_R$) denotes the boundary components of $W$ arising from left- (right-)hand sides of curves of $\vgamma$.
\end{corollary}

\begin{proof}
The hypothesis allows us to apply Theorem \ref{theorem:exttv_oriented} with $\vbeta$ empty. The ``only if'' part of this statement is immediate.
For the ``if'' part, we need to check that no curves are added to $\valpha$ during the extension process. Since $\vbeta$ is empty, it does not include any singletons, so neither Lemma \ref{lem:addcurve_removesingleton} nor Lemma \ref{lem:nosingletons_extendtofill} add any new curves to $\valpha$. Thus, we obtain a multicurve $\vbeta$ such that $\valpha$ and $\vbeta$ are coherent and filling, hence $\TV(\valpha, \vbeta)$ gives a translation surface on which $\valpha$ constitutes the entire horizontal foliation.
\end{proof}

By assigning orientations, we can also use Theorem \ref{theorem:exttv_oriented} to deduce our result about unoriented multicurves, which was stated in the Introduction as Theorem \ref{mainthm:parallel}.

\begin{corollary}
If $\alpha$ and $\beta$ are a pair of (unoriented) multicurves on $S$ with no curves in common, then they can be realized as a pair of parallel multicylinders if and only if the following hold:
\begin{enumerate}
    \item The multicurves $\alpha$ and $\beta$ are coherently orientable.
    \item No single component of $\alpha$ separates $S \setminus \beta$.
    \item No single component of $\beta$ separates $S \setminus \alpha$.
\end{enumerate}
\end{corollary}

\begin{proof}
Lemma \ref{1 parallel implies coherently orientable}
and Corollary \ref{cor:top version of obstruction} together prove that these conditions are necessary.

To prove that they are sufficient, let $\alpha$ and $\beta$ be a pair of unoriented multicurves satisfying the conditions.
Let $\alpha_0 = \alpha \setminus \singletons_{\beta} (\alpha)$ denote the non-singletons of $\alpha$ and let $\beta_0 \subset \beta$ be defined similarly.
By condition (1), we can orient $(\alpha_0, \beta_0)$ such that $(\valpha_0, \vbeta_0)$ are coherent.
It remains to orient the singletons so that \eqref{noobstruction} holds.

As in the proof of Lemma \ref{lem:addcurve_removesingleton}, consider the dual (undirected) graph $G$ of the multicurve $\singletons_{\beta} (\alpha)$ on $S \setminus \beta$.
By condition (2), each component $G_0$ of $G$ is 2-edge-connected, that is, removing any edge does not separate $G_0$.
Thus, by Robbins's Theorem
there exists a choice of orientation for the edges of $G_0$ so that the resulting directed graph is strongly connected.
Choose one such orientation for each $G_0$, and orient each $a \in \singletons_{\beta} (\alpha)$ so that there is an edge from $W_i \subset S \setminus (\beta \cup \singletons_{\beta} (\alpha))$ to $W_j$ if $W_i$ is on the left and $W_j$ lies on the right of $\vec{a}$.
Orient the singletons of $\beta$ similarly.

Now let $ \vgamma \subseteq \valpha \cup \vbeta$ and let $W$ be a complementary subsurface of $S \setminus \vgamma$.
Let us prove that $A_L \neq \emptyset$ if and only if $A_R \neq \emptyset$; 
the proof for $B_L$ and $B_R$ is completely analogous.
If $\vgamma \cap \valpha_0$ is nonempty, then $W$ must meet some curve of $\vbeta_0$. Consider an arc of $\vbeta_0|_W$; since $\vbeta_0$ always crosses $\valpha_0$ from right to left, this means that it must enter $W$ when it meets a curve of $A_L$ and leave $W$ when it meets a curve of $A_R$. Thus $A_L$ and $A_R$ must both be nonempty.

Otherwise, suppose that $\vgamma \cap \valpha_0 = \emptyset$.
Of course, if $\vgamma \cap \singletons_{\vbeta}(\valpha)$ is also empty then we are done, so assume there is some boundary component of $W$ that is a singleton of $\alpha$.
In this case, then we see that $W$ is a union of pieces of $S \setminus (\beta \cup \singletons_{\vbeta}(\alpha))$, i.e., it is a component of a cut of $G_0$.
But now since $G_0$ is strongly connected, there are edges both entering and exiting this component, hence both $A_L$ and $A_R$ are nonempty.

Therefore, we have shown that we can orient $\alpha$ and $\beta$ so that \eqref{noobstruction} holds; applying Theorem \ref{theorem:exttv_oriented} completes the proof of the Corollary.
\end{proof}

\section{Coherence and grafting} \label{sec:total-coherence-and-grafting}
In the previous section, we characterized when two (un)oriented multicurves could be realized simultaneously as (parallel) directional multicyinders.
In this one, we relax this requirement, giving necessary and sufficient conditions to simultaneously realize two multicurves as multicylinders with just one being (parallel) directional.
As in the previous section, we prove a stronger statement for multicurves with prescribed orientations; compare Theorem \ref{theorem:grafting} below.

See Figure \ref{fig:theoeg} for an example showing that the conditions of Theorem \ref{mainthm:parallel} are too strong if one only requires a single multicylinder to be parallel.

\begin{figure}[ht]
    \centering
    \includegraphics[scale=.18]{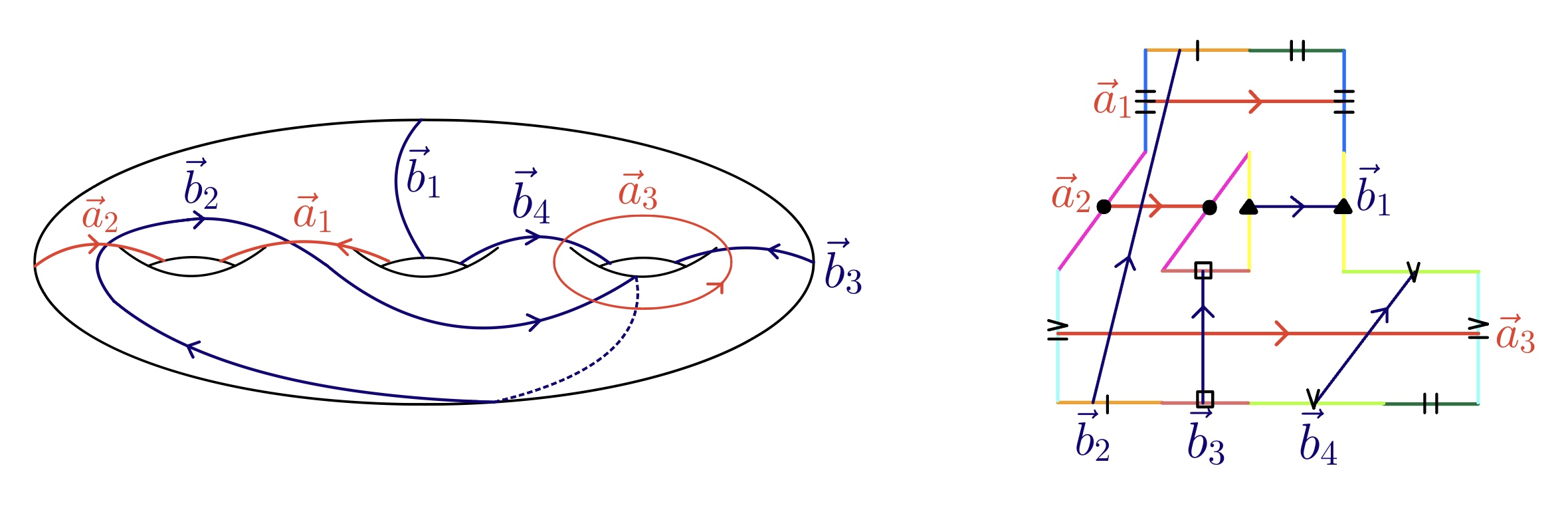}
    \caption{A pair of coherent multicurves that cannot be jointly realized as parallel multicylinders, together with a realization with one of the multicylinders parallel. Observe that even if one repartitions and adds the singleton of $\beta$ to $\alpha$, then the pair is still not realizable as a pair of parallel multicylinders.}
    \label{fig:theoeg}
\end{figure}

As in the previous section, we will fix the convention that given two coherent multicurves $\valpha$ and $\vbeta$, the curves of $\vbeta$ cross the curves of $\valpha$ from right to left.

\subsection{Geodesics on translation surfaces}
We begin by recording a simple lemma about how cylinders intersect other geodesics on translation surfaces. We recall that if a curve $b$ is {\em not} realized as a cylinder on $(X, \omega)$, then it has a unique geodesic representative which is a concatenation of saddle connections.

One of the handy things about cylinders is that their core curves are always in minimal position with respect to other geodesics.

 \begin{lemma}
\label{minimal position}
Let $a$ be a nonsingular core curve of a cylinder on some translation surface $(X,\omega)$ and let $b$ be the geodesic representative of a non-cylinder curve on $(X,\omega)$.
Then $a$ and $b$ are in minimal position.
\end{lemma}

In particular, if the geometric intersection number of the isotopy classes of $a$ and $b$ is $0$, then the actual geodesics $a$ and $b$ are disjoint.

\begin{proof}
Because $a$ is the core curve of a cylinder, it contains no saddle connections. The curve $b$ is a concatenation of saddle connections, so $a$ and $b$ are transverse.
 
Assume for contradiction that $a$ and $b$ are not in minimal position. Then $a$ and $b$ bound a bigon \cite[Proposition 1.7]{FarbMarg}.
Consider the two arcs $c_a \subset a$ and $c_b \subset b$ constituting edges of this bigon; since $a$ and $b$ are transverse, $c_a \neq c_b$.
Now $c_a$ and $c_b$ are isotopic rel endpoints, so the surgered curve $(b \setminus c_b) \cup c_a$ is isotopic to $b$. Because $b$ is the unique shortest representative of its isotopy class, $c_b$ must be strictly shorter than $c_a$. However, this implies that the curve $(a \setminus c_a) \cup c_b$ is isotopic to and shorter than $a$, which is a contradiction.
Thus $a$ and $b$ are in minimal position.
\end{proof}

\begin{remark}
Lemma \ref{minimal position} is not true for all pairs of geodesics on $(X, \omega)$.
For example, there may be curves $a$ and $b$ which have geometric intersection number $0$ but whose geodesic representatives share a saddle connection. Compare with the discussion in \cite[Section 3]{Rafi}.
\end{remark}

If $I \subset [0, 2\pi)$ is an interval (with any mix of conditions at its endpoints), then we say that $\vec{s}$
is an $I$-{\em saddle connection} if $\text{arg}(\hol(\vec{s})) \in I$.

We now state and prove an analogue of Lemma \ref{1 parallel implies coherently orientable} for non-cylinder curves.

\begin{corollary}\label{cor:realize directional implies angles}
Suppose that $(X, \omega)$ is a horizontally periodic translation surface and let $\valpha$ denote the core curves of the horizontal cylinders, oriented in the $+x$ direction.
Let $\vb$ be any oriented curve coherent with $\valpha$ and suppose that $\vb$ is not realized as a cylinder on $(X,\omega)$.
Then the geodesic representative of $\vb$ is a concatenation of $[0,\pi]$ saddle connections.
\end{corollary}
\begin{proof}
Lemma \ref{minimal position} says $\valpha$ and $\vb$ realized on $(X,\omega)$ are in minimal position, so there are no bigons. Thus, the algebraic and geometric intersection numbers of the isotopy classes of $\valpha$ and $\vec{b}$ agree with those of their geodesic representatives.
Since $(X, \omega)$ is horizontally periodic, $\valpha$ meets any non-horizontal saddle connection.
Thus, since $\valpha$ and $\vec{b}$ are coherent, any non-horizontal saddle connection comprising $\vb$ must be a $(0,\pi)$-saddle connection. This completes the proof.
\end{proof}

\subsection{Flat grafting}
For the remainder of this section, when we refer to a curve on $(X,\omega)$, we mean its geodesic representative.

Given a pair of coherent multicurves $(\valpha,\vbeta)$, our strategy to realize them as multicylinders is to find a translation surface $(X,\omega)$ on which $\valpha$ is a horizontal multicylinder, and then deform $(X,\omega)$ to make each curve of $\vbeta$ into a cylinder while ensuring $\valpha$ remains horizontal.
A deformation that accomplishes the second step is {\em horizontal grafting}, which was introduced in \cite{fu2018flat}.
We direct the reader to the original paper \cite{fu2018flat} for a more thorough discussion of this procedure.

\begin{construction} [Horizontal Grafting]
\label{construction:grafting}
Let $(X,\omega)$ be a translation surface and suppose that $c$ is a geodesic curve on $(X,\omega)$ not isotopic to the core curve of any cylinder.
Let $\{s_1,...,s_l\}$ denote the saddle connections appearing in $c$, counted without multiplicity.
Cutting $(X,\omega)$ along $c$ results in a translation surface with piecewise-geodesic boundary, two geodesic segments for each $s_i$.

Suppose for the moment that no $s_i$ is horizontal. Then the {\em grafting} of $(X,\omega)$ along $c$ (by some distance $t$) is obtained by gluing each copy of $s_i$ in $X \setminus c$ to a parallelogram with sides $\langle t, 0 \rangle$ and $s_i$, then gluing the horizontal sides of the parallelograms together according to how $c$ runs along the $s_i$.
See Figure \ref{fig: deformation iteration}.

If some $s_i$ is horizontal, then there are two different choices for how to define the grafting of $(X,\omega)$ along $c$.
In either case, the grafted surface is obtained as a limit of rotating $(X,\omega)$ (either counterclockwise or clockwise) by a small bit, grafting, and then rotating back.
This can be thought of as ``shearing'' the surface along that saddle and identifying segments of $s_i$ with the horizontal edges of non-horizontal parallelograms.
\end{construction}

\begin{remark}\label{rmk:graftingandmarking}
When $(X,\omega)$ has a marking $\marking: S \to X$, the grafted surface $(X',\omega')$ inherits a marking $\marking': S \to X'$ since the grafting procedure takes place entirely in a neighborhood of $c$.
However, depending on the angle that $c$ makes at each cone point, grafting may split apart cone points of $(X,\omega)$, changing which stratum it lives in while preserving the genus of the surface.
Therefore, there is no canonical way to mark both the translation surface {\em and its zeros} that is consistent under grafting.
This is one of the reasons that we have decided to focus on the realizability of curves on a closed surface as cylinders, as opposed to considering curves relative to zeros.
\end{remark}

We record below a number of properties of grafting, proofs of which can be found in \cite{fu2018flat}.

\begin{lemma}
\label{Fu properties of grafting}
Let all notation be as above. Orient $c$ arbitrarily, inducing an orientation on each $s_i$.
\begin{enumerate}
\item If $c$ is a concatenation of $I$-saddle connections, where $I$ is a proper closed subinterval of $[0,\pi]$, then there is a choice of horizontal grafting such that the grafted surface is a translation surface.
\item If $c$ contains at least one non-horizontal saddle connection, then there exists some distance $t$ such that $c$ is realized as a cylinder on the grafting of $(X,\omega)$ along $c$ by $t$.
\end{enumerate}
\end{lemma}
\begin{proof}[Sketch of (1)]
If $I$ contains neither $0$ nor $\pi$, then there is only one choice of horizontal grafting. 
If it contains $0$, then we choose the counterclockwise direction, and if it contains $\pi$, the clockwise direction.
That the resulting surface is a translation surface was proved in \cite{fu2018flat}.
\end{proof}

Disjointness of each curve in $\vbeta$ allows us to graft along each curve in $\vbeta$ without interfering with the cylindrical property of the previously grafted curves in $\vbeta$. This is because grafting is a local procedure.

\begin{lemma}
\label{local_procedure}
Let $b_1$ be a core curve of a cylinder and $b_2$ a non-cylinder curve on some translation surface $(X,\omega)$. If $\geomintersect(b_1,b_2) = 0$, then $b_1$ is realized as a cylinder with the same slope as before on any surface obtained by grafting along $b_2$.
\end{lemma}
\begin{proof} 
By Lemma \ref{minimal position}, $b_1$ and $b_2$ are in minimal position. Hence, since $\geomintersect(b_1,b_2) = 0$, their geodesic representatives do not intersect. 
Recall that when we graft horizontally along $b_2$, we cut along $b_2$ and glue in a parallelogram $p_i$ adjacent to each saddle connection $s_i$ on $b_2$. In particular, $(X, \omega) \setminus b_2$ is isometric to the complement of the $p_i$'s in the grafted surface.
Thus, $b_1$ remains a cylinder on any grafting of $b_2$.
\end{proof}


Because curves in $\vbeta$ are disjoint, we can now take the translation surface which realizes $\valpha$ as the horizontal foliation and graft horizontally along each curve of $\vbeta$ until $(\valpha,\vbeta)$ are multicylinders. This is shown in the backwards direction of Theorem \ref{theorem:grafting}.

In what follows, a triple of oriented multicurves $(\vgamma_1,\vgamma_2,\vgamma_3)$ is coherent if every pair $(\vgamma_i,\vgamma_j)$ is coherent.

\begin{figure}
    \centering
    \def\svgwidth{\columnwidth}
    \includegraphics[scale=.18]{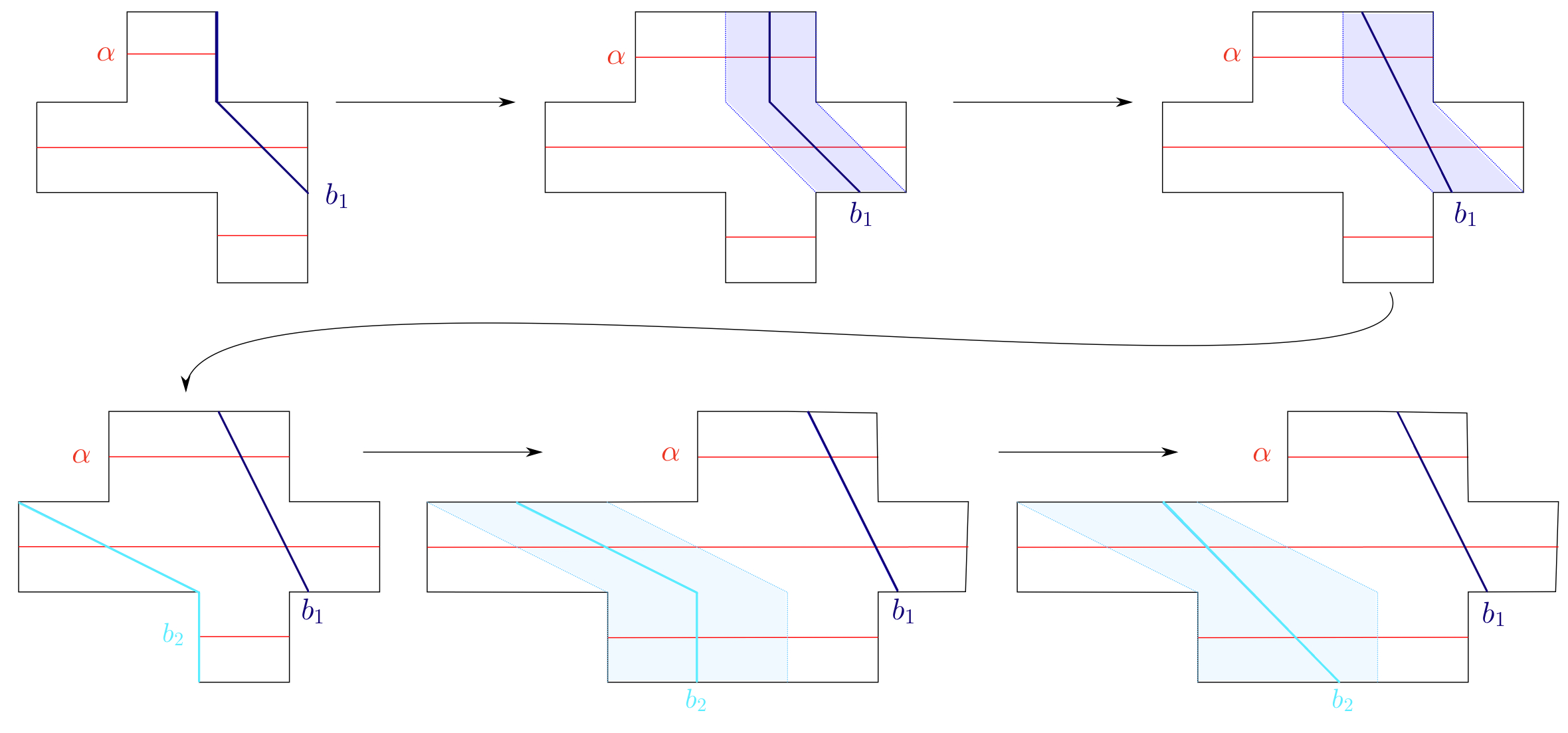}
    \caption{Example of grafting along two curves of $\beta$, $b_1$ and $b_2$ to realize them as cylinders on a translation surface where $\alpha$ is realized as horizontal cylinders}
    \label{fig: deformation iteration}
\end{figure}

\begin{theorem}
\label{theorem:grafting}
Let $(\valpha , \vbeta)$ be a pair of multicurves on surface $S$ which contains no separating curves.
Then $(\valpha , \vbeta)$ can be realized as a pair of multicylinders with $\valpha$ a directional multicylinder on some translation surface if and only if there exists a filling pair $(\valpha', \vec{\gamma})$ such that $\valpha \subset \valpha'$ and $(\valpha', \vgamma, \vb)$ is coherent for each $\vb \subset \vbeta$.
\end{theorem}
\begin{proof}
Consider some translation surface which realizes $\valpha$ as a directional multicylinder and $\vbeta$ as an arbitrary multicylinder.
Postcomposing with an element of $\GL_2(\RR)$ as necessary, we may assume that each curve of $\valpha = \vec{a}_1\cup...\cup \vec{a}_m$ is horizontal and oriented in the +$x$-direction. Moreover, we can shear the surface enough so that the angle of the holonomy of each curve of $\vbeta$ lies in the set $[0,\frac{\pi}{2}) \cup [\pi,\frac{3\pi}{2})$. Call this new surface $(X,\omega)$.

Suppose the curves of $\vbeta$ which are core curves of horizontal cylinders on $(X,\omega)$ are $\vh_1,...,\vh_k$.
As in Lemma \ref{lemma:tv-directional-equivalence}, we can choose a local period coordinate chart for the ambient stratum and cut out a non-zero $\mathbb{R}$-linear subspace $V$ of $H^1(X, \text{Zeros}(\omega);\mathbb{C})$ by stipulating that
\[\text{Im}(\text{hol}(\vec{a}_i)) = 0 \text{ and } \text{Im}(\text{hol}(\vec{h}_j)) = 0\]
for all $i = 1,...,m$ and $j = 1,...,k$.
Recall by Lemma \ref{lem:cylinderspersist} there is some relatively open subset $U \subset V$ containing $(X, \omega)$ in the stratum in which each element of $U$ realizes $(\valpha, \vbeta)$ as a pair of multicylinders. 
Furthermore, $U$ can be chosen sufficiently small enough so that $\vbeta$ still points in direction $[0,\frac{\pi}{2}) \cup [\pi,\frac{3\pi}{2})$ on every surface.
Since $V$ is cut out by integral-linear equations, rational points $H^1(X, \text{Zeros}(\omega); \QQ \oplus i \QQ)$ are dense in $V$ and so we can find some square-tiled surface 
\[(X', \omega') \in H^1(X, \text{Zeros}(\omega);\mathbb{Q} \oplus i\mathbb{Q}) \cap U.\]
Since square-tiled surfaces are vertically and horizontally periodic, we may consider the filling pair $(\valpha', \vec{\gamma})$ where $\valpha'$ are the horizontal core curves (oriented in the $+x$ direction) and $\vec{\gamma}$ are the vertical core curves (oriented in the $+y$ direction). In particular, $\valpha \subset \valpha'$ and each curve $\vb \in \vbeta$ is coherent with $\valpha'$ and $\vec{\gamma}$.
\bigskip

Conversely, assume such a pair $(\valpha', \vgamma)$ exists. If some curves of $\vbeta$ do not intersect the curves of $\valpha'$, then we can re-partition the pair 
\[\valpha'':= \valpha' \cup 
\singletons_{\valpha'}(\vbeta)
\text{ and }
\vbeta' := \vbeta - 
\singletons_{\valpha'}(\vbeta)\]
so that every curve in $\vbeta'$ intersects $\valpha''$. Note that though $(\valpha'',\vgamma)$ may not be coherent, e.g. when $\vb_{i_j}$ intersects $\vgamma$ in the opposite direction as $\valpha'$ does, the pair is still coherently orientable, as we can simply flip the orientations on each $\vb_{i_k}$.

By Construction \ref{construction:tv}, there is some square-tiled surface $(X , \omega)$ that realizes $\valpha''$ and $\vgamma$ as its horizontal and vertical cylinders, respectively. Moreover, we may assume $\valpha'$ points in the $+x$ direction and $\vgamma$ the $+y$ direction. Coherence with $(\valpha', \vec{\gamma})$ combined with disjointness from $\valpha'' - \valpha'$ ensures that each curve $\vb \subset \vbeta'$ is coherent with $\valpha''$ and $\vec{\gamma}$.
Corollary \ref{cor:realize directional implies angles} applied using $\valpha''$ as the horizontal cylinders, followed by rotating the surface by $\pi/2$ and re-applying the Corollary, implies that they geodesic representatives of $\vb$ on $(X , \omega)$ are concatenations of either only $[0,\frac{\pi}{2}]$-, $[\frac{\pi}{2}, \pi]$- , $[\pi,\frac{3\pi}{2}]$-, or $[\frac{3\pi}{2}, 2\pi]$-saddle connections (or a cylinder pointing in one of those intervals).

Suppose that $\vbeta' = \vb_1 \cup ... \cup \vb_m$. We will first show $(\valpha',\vb_1)$ can be simultaneously realized as cylinders on some translation surface using horizontal grafting (See Construction \ref{construction:grafting}). If this is already the case on $(X , \omega)$, we then proceed to $(\valpha, \vb_1 \cup \vb_2)$. By Lemma \ref{Fu properties of grafting} (1), horizontal grafting along $\vb_1$ yields a translation surface. Because no curve in $\vbeta'$ is disjoint from $\valpha''$, all curves of $\vbeta'$ have a non-horizontal saddle connection. By Lemma \ref{Fu properties of grafting} (2), we can graft horizontally along $\vb_1$ on $(X,\omega)$ some sufficient amount to obtain a translation surface $(X_1,\omega_1)$ that realizes $\vb_1$ as the core of a cylinder. Furthermore, $\valpha''$ remains as the cores of the complete set of horizontal cylinders on $(X_1,\omega_1)$ because horizontal grafting preserves the horizontal foliation.

Next, we will realize $\vb_2$ on $(X_1,\omega_1)$ and re-iterate this procedure. Assume $\vb_2$ is not the core of a cylinder, or else we can move onto $\vb_3$. By Lemma \ref{minimal position}, $\vb_2$ is disjoint from $\vb_1$ because $\geomintersect(\vb_1, \vb_2)=0$ and $\vb_1$ and $\vb_2$ are in minimal position.
We note that the vertical foliation of the grafted surface is no longer $\vgamma$.
However, $\vb_1$ and $\vb_2$ did not intersect transversely on $(X,\omega)$ and $(X,\omega)$ is isometric to $(X_1,\omega_1)$ away from $\vb_1$ and the inserted parallelograms, so the saddle connections comprising $\vb_2$ on $(X_1,\omega_1)$ all point in the same directions as they did on $(X,\omega)$.

Again, we can graft horizontally along $\vb_2$ enough to obtain a translation surface $(X_2,\omega_2)$ that realizes $\vb_2$ as the core of a cylinder. Moreover, by Lemma \ref{local_procedure}, $\vb_1$ is still the core of a cylinder on $(X_2,\omega_2)$. Once again, $\valpha''$ remains as the cores of the complete set of horizontal cylinders.

We continue iterating this process for each non-cylinder curve in $\vbeta'$ to realize $\vbeta$ as a multicylinder on some translation surface $(X_N,\omega_N)$ where $\valpha$ is a directional multicylinder. 
\end{proof}

\section{Pairwise coherence} \label{sec:pairwise-coherent-multicurves-as-cylinders}

Recall that two oriented multicurves $\valpha$ and $\vbeta$ on $S$ are said to be {\em pairwise coherent} if for any pair of curves $\vec{a_i} \subset \valpha$ and $\vec{b_j} \subset \vbeta$, the pair $(\vec{a_i}, \vec{b_j})$ is coherent.
Lemma \ref{1 parallel implies coherently orientable} implies that pairwise coherence is necessary for two multicurves to be simultaneously realizable as multicylinders.

In this section, we provide two examples of a pair of pairwise coherent, but not coherently orientable, multicurves $\valpha$ and $\vbeta$ on some $S$.
Our first Example \ref{ex:multicyl_notcoherent} is an instance where $\valpha$ and $\vbeta$ can be simultaneously realized as a pair of multicylinders on some translation surface, whereas in Example \ref{ex:pairwisecoherent_notrealizable}, they cannot.

\begin{example}\label{ex:multicyl_notcoherent}
Consider the multicurves $\valpha = \vec{a}_1 \cup \vec{a}_2$ and $\vbeta = \vec{b}_1 \cup \vec{b}_2$ on $S_2$ as shown to the right in Figure \ref{fig: not coherent}. The picture to the left is a translation surfaces which realizes $\valpha$ and $\vbeta$ as cylinders. However, by observation of all possible assignments of orientation to $\vec{a}_1$, $\vec{a}_2$, $\vec{b}_1$, and $\vec{b}_2$, we see there is no combination of orientations that makes $\valpha$ and $\vbeta$ coherent.
\end{example}

\begin{figure}[h]
    \centering
    \subcaptionbox{Polygonal representation}{\includegraphics[width=0.4\textwidth]{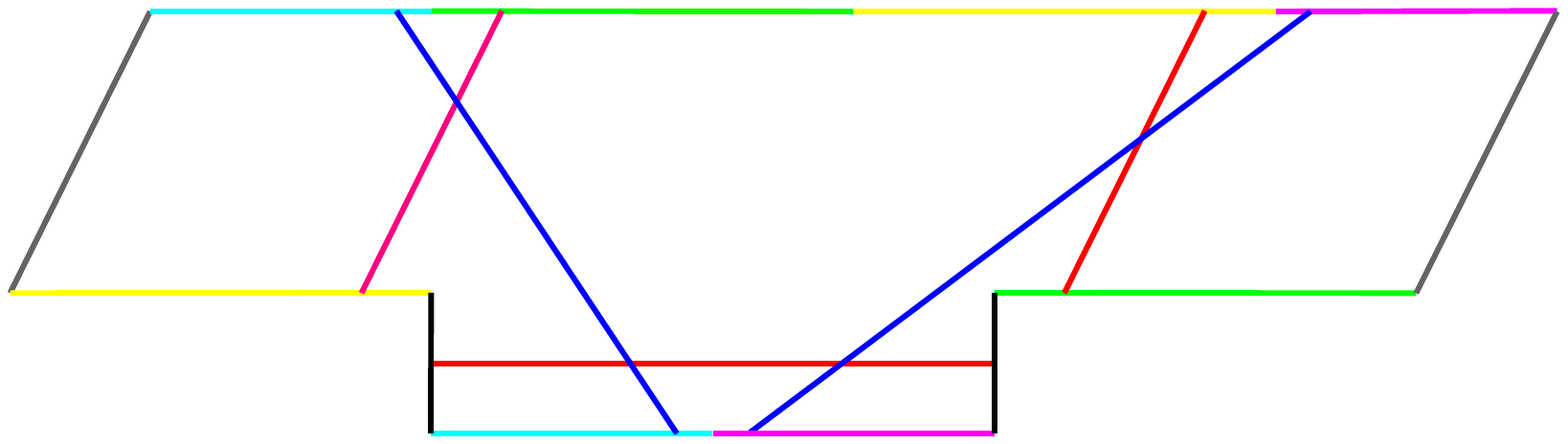}}
    \hfill
    \subcaptionbox{Topological representation}{\includegraphics[width=0.4\textwidth]{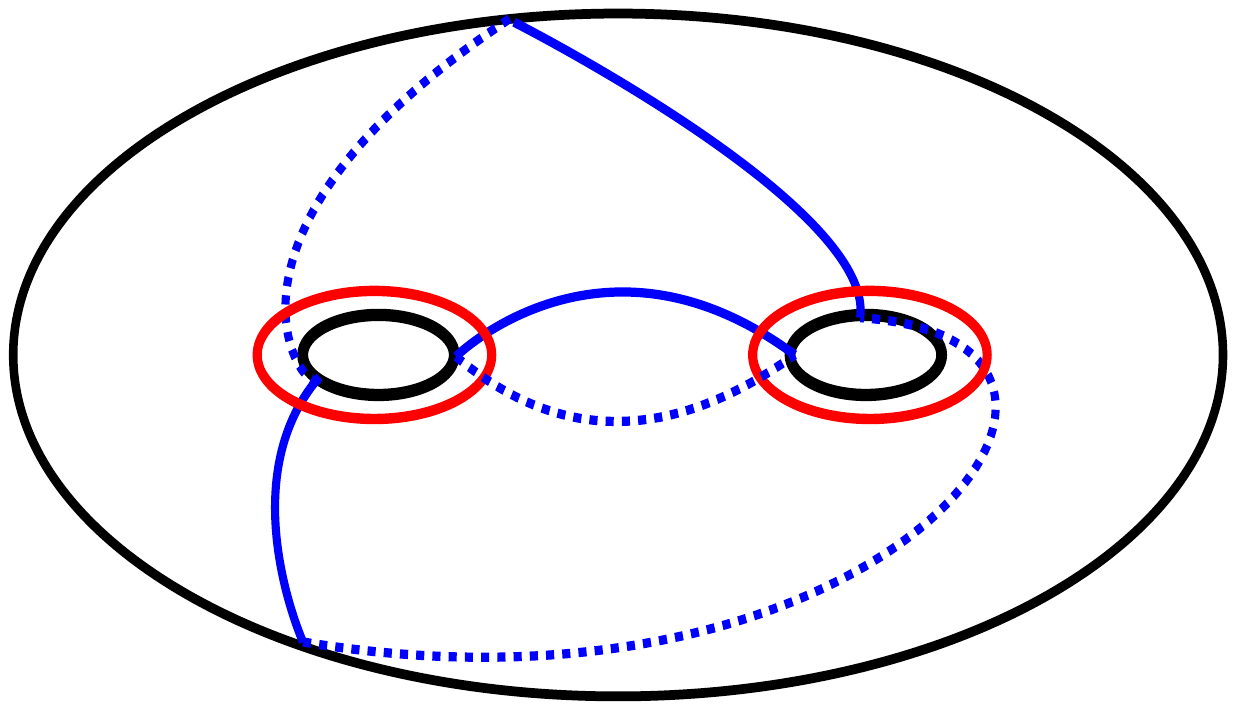}}
    \caption{A set of multicurves which are not coherently orientable are still simultaneously realizable as cylinders on this translation surface. }
    \label{fig: not coherent}
\end{figure}

\begin{example}\label{ex:pairwisecoherent_notrealizable}
Next, we provide an example of two pairwise coherent multicurves that cannot be realized as a pair of multicylinders on any translation surface.
Let multicurves
$\valpha = \vec{a_1} \cup \vec{a_2} \cup \vec{a_3} \cup \vec{a_4}$
and
$\vbeta= \vec{b_1} \cup \vec{b_2} \cup \vec{b_3}$ on $S_3$
be as in Figure \ref{fig:pairwise-nonrealizable}. 
These multicurves are pairwise coherent but are not coherently orientable. 

\begin{figure}[h]
    \centering
    \begin{subfigure}[c]{0.6\textwidth}
      \centering
      \def\svgwidth{\columnwidth}
      \includegraphics[width=.8\textwidth]{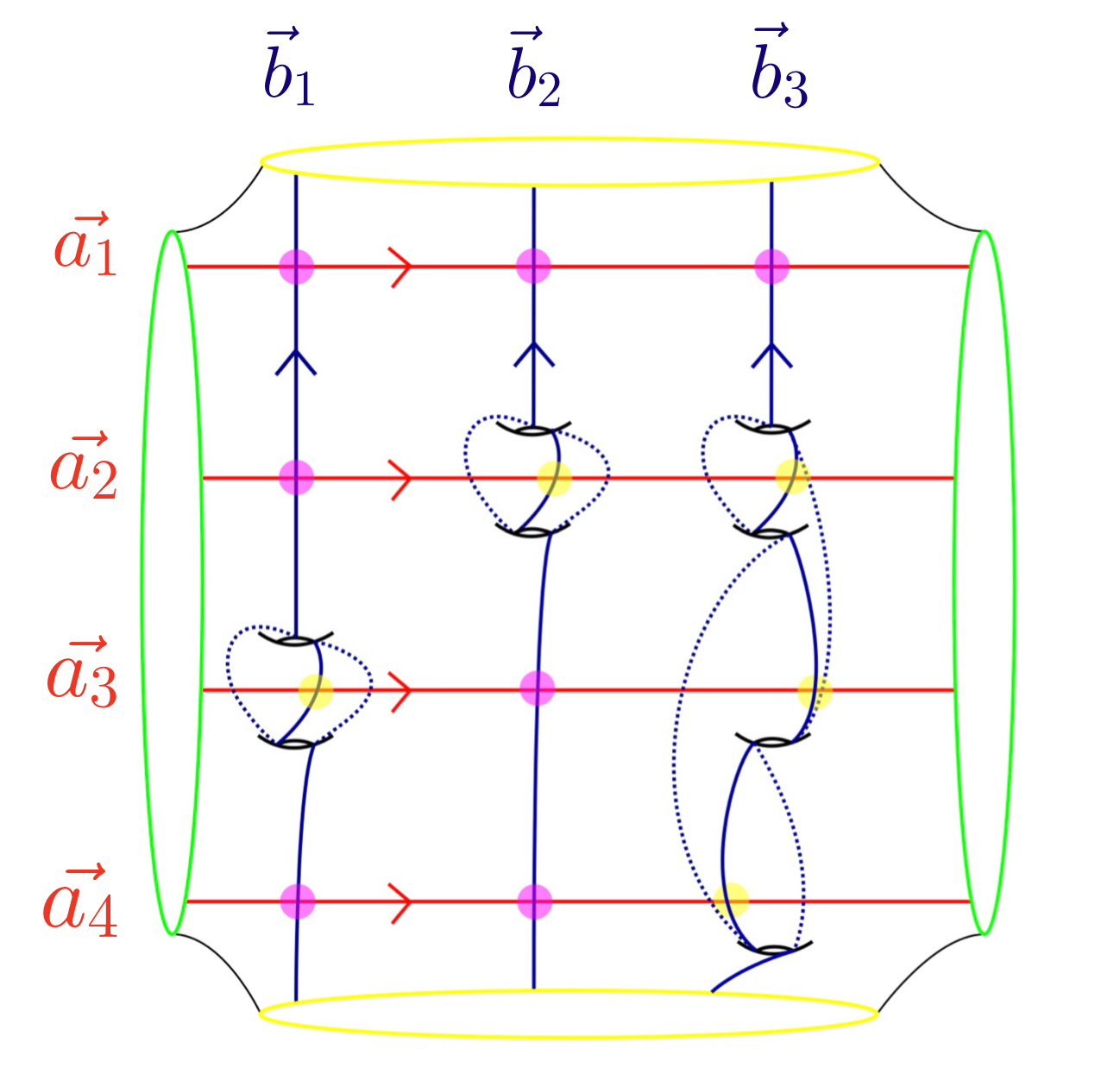}
    \end{subfigure}
\hspace{1ex}
    \begin{subfigure}[c]{0.2\textwidth}
      \centering
      \Large
      \begin{tabular}{|c|c|c|c|} 
        \hline
        & $\vec{b}_1$ & $\vec{b}_2$ & $\vec{b}_3$\\ 
        \hline
        $\vec{a}_1$ & + & + & +\\ 
        \hline
        $\vec{a}_2$ & + & $-$ & $-$\\ 
        \hline
        $\vec{a}_3$ & $-$ & + & $-$ \\
        \hline
        $\vec{a}_4$ & $+$ & + & $-$ \\
        \hline
      \end{tabular}
    \end{subfigure}
    \caption{In the table to the right, the intersection pattern between the two multicurves is given by $\algintersect(\va_i,\vb_j)$, where + indicates $\algintersect(\va_i,\vb_j) > 0$ and $-$ indicates $\algintersect(\va_i,\vb_j) < 0$.}
    \label{fig:pairwise-nonrealizable}
\end{figure}

Suppose by contradiction there is some translation surface $(X,\omega)$ that realizes $\valpha$ and $\vbeta$ as cylinders. Thus, there exists an assignment of angles, $\{\anga_1, \anga_2, \anga_3, \anga_4, \angb_1, \angb_2, \angb_3\}$ relative to the positive $x$-axis recording the angles of the cylinders.
Furthermore, the direction in which $\vec{a}_i$ contacts $\vec{b}_j$ agrees with the algebraic intersection numbers on $S_3$ shown in Figure \ref{fig:pairwise-nonrealizable}. As we shall see, it is impossible to find any such collection of angles $\{\anga_1, \anga_2, \anga_3, \anga_4, \angb_1, \angb_2, \angb_3\}$ that agrees with the prescribed intersection numbers.

When assigning angles to curves, without loss of generality, we can choose any one curve to begin with and give it any angle. So, let us set $\angb_1 = 0$.
Then $\anga_1 \in (\pi, 2\pi)$ and $\anga_2 \in (\pi, 2\pi)$ because $\algintersect(\anga_1,\angb_1) > 0$ and $\algintersect(\anga_2,\angb_1) > 0$.
We note that $\anga_1$ cannot equal $\anga_2$, for if it did then $\vb_2$ could have constant slope on $(X,\omega)$ because it is not coherent with the foliation in the direction of $\anga_1=\anga_2$.
Thus, we will break up the rest of this proof into two cases, (1) $\anga_1 > \anga_2$ and (2) $\anga_1 < \anga_2$, each with two subcases. See Figure \ref{fig:pairwise-nonrealizable} for illustrations.
 
\begin{enumerate} 
\item We begin with the assumption that $\anga_1 > \anga_2$.
 
In this case, $\algintersect(\va_1,\vb_2) >0 $ tells us that $\angb_2 \in \big( (\anga_1, 2 \pi) \cup [0, \anga_1 - \pi) \big)$ and $\algintersect(\va_2,\vb_2) < 0$ means that $\angb_2 \in (\anga_2 - \pi, \anga_2)$. Together, we have that $$\angb_2 \in \big((\anga_1, 2 \pi) \cup [0, \anga_1 - \pi)\big) \cap (\anga_2 - \pi,  \anga_2)$$ which simplifies to
\[\angb_2 \in (\anga_2 - \pi, \anga_1 - \pi).\]
Since $\vb_3$ has the same algebraic intersections with $\va_1$ and $\va_2$, we get the same conclusion for $\angb_3$.

We can rule out the case that $\angb_2 = \angb_3$ as above, as this would imply that $\va_3$ would not be coherent with the foliation in the direction of $\angb_2$ and $\angb_3$. Since $\angb_2$ and $\angb_3$ lie in the same interval, we have the following two subcases to examine:
\begin{enumerate}
\item Suppose first that $\angb_3 > \angb_2$.  Then, because $\algintersect(\va_3,\vb_2) > 0$ and  $\algintersect(\va_3,\vb_3) < 0$, we have that $$\anga_3 \in \big( (\angb_2 + \pi, 2\pi) \cup (0, \angb_2) \big) \cap (\angb_3, \angb_3 + \pi)$$ which simplifies to 
\[\anga_3 \in (\angb_2 + \pi,\angb_3 + \pi) \subset (\pi, 2\pi).\]
However, this implies that $\algintersect(\va_3, \vb_1) > 0$, whereas the table in Figure \ref{fig:pairwise-nonrealizable} states that $\algintersect(\va_3,\vb_1) < 0$, resulting in a contradiction.

\item Now, suppose $\angb_3 < \angb_2$. We have that $\algintersect(\va_4,\vb_2) > 0$ and $\algintersect(\va_4,\vb_3) < 0$. Therefore, $$\anga_4 \in \big( (\angb_2 + \pi, 2 \pi) \cup (0, \angb_2) \big) \cap (\angb_3, \angb_3 + \pi)$$ which simplifies to 
\[\anga_4 \in (\angb_3, \angb_2) \subset (0, \pi).\]
However, we have arrived at a contradiction again since this implies that $\algintersect(\va_4 , \vb_1) < 0$, whereas Figure \ref{fig:pairwise-nonrealizable} states that $\algintersect(\va_4,\vb_1) > 0$.
\end{enumerate}
\item For the second case, we assume $\anga_1 < \anga_2.$ Because $\algintersect(\va_1,\vb_2) > 0 $ and $\algintersect(\va_2,\vb_2) < 0$, we see that $$\angb_2 \in \big( (\anga_1, 2 \pi) \cup [0, \anga_1 - \pi) \big) \cap (\anga_2 - \pi, \anga_2)$$ which simplifies to $$\angb_2 \in (\anga_1, \anga_2).$$ By identical reasoning using $\angb_3$ in place of $\angb_2$, we have that $\angb_3$ lies in the same interval. 
We again split into two cases:
    \begin{enumerate}
        \item If $\angb_3 > \angb_2$, then because $\algintersect(\va_4,\vb_2) > 0$ and  $\algintersect(\va_4,\vb_3) < 0$, we have that $$\anga_4 \in (\angb_2 - \pi, \angb_2) \cap \big( (\angb_3, 2\pi) \cup (0, \angb_3 - \pi) \big)$$ which simplifies to $$\anga_4 \in (\angb_2 - \pi, \angb_3 - \pi).$$ This result implies that $\anga_4 \in (0, \pi)$ and hence $\algintersect(\va_4, \vb_1) < 0$, however this contradicts Figure \ref{fig:pairwise-nonrealizable} where $\algintersect(\va_4, \vb_1) > 0$.
        \item If $\angb_3 < \angb_2$, then because $\algintersect(\vec{a_3},\vec{b_2}) > 0$ and  $\algintersect(\vec{a_3},\vec{b_3}) < 0$, we have that $$\anga_3 \in (\angb_2 - \pi, \angb_2) \cap \big( (\angb_3, 2 \pi) \cup (0, \angb_3 - \pi) \big)$$ which simplifies to $$\anga_3 \in (\angb_3, \angb_2).$$ This result implies that $\anga_3 \in (\pi, 2\pi)$ and hence $\algintersect(\va_3, \vb_1) > 0$, contradicting Figure \ref{fig:pairwise-nonrealizable} where $\algintersect(\va_3, \vb_1) < 0$.
    \end{enumerate}
    
\end{enumerate}
We conclude no assignment of angles to $\valpha$ and $\vbeta$ on $(X,\omega)$ is compatible with the specified algebraic intersection numbers. Therefore, $\valpha$ and $\vbeta$ cannot be realized as a pair of multicylinders on any translation surface.
\end{example}

We observe that all of the curves in the example above are homologically independent, so they will satisfy \eqref{noobstruction}. Moreover, our argument used only the intersection numbers of these curves, not their precise configuration on the surface, so this obstruction can be detected at the level of homology.

\begin{figure}[ht]
\centering
    \begin{subfigure}{\linewidth}
  \includegraphics[width=.5\linewidth]{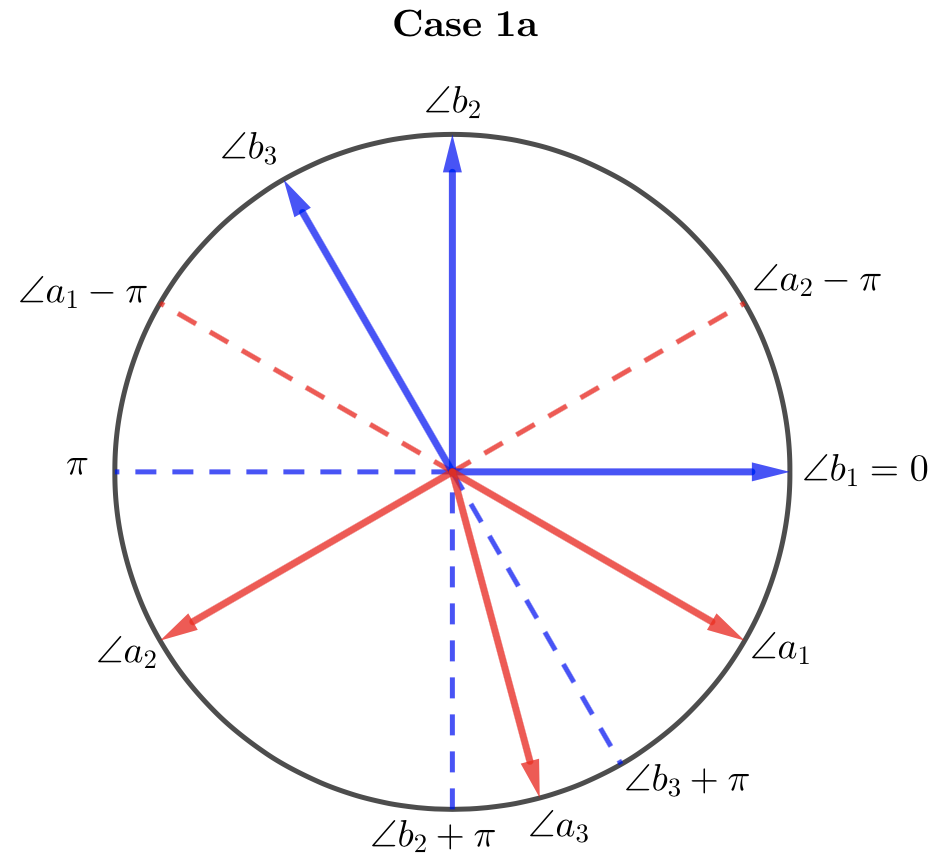}\hfill
  \includegraphics[width=.5\linewidth]{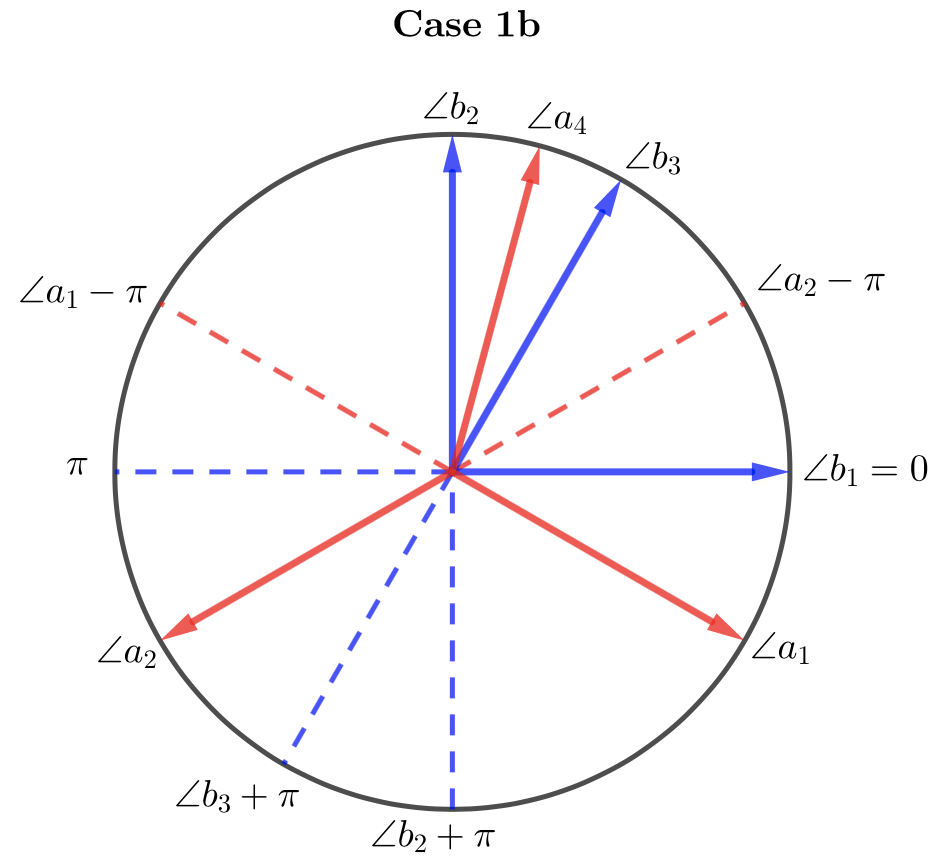}\hfill
  \end{subfigure}\par\bigskip \bigskip
  \begin{subfigure}{\linewidth}
  \includegraphics[width=.5\linewidth]{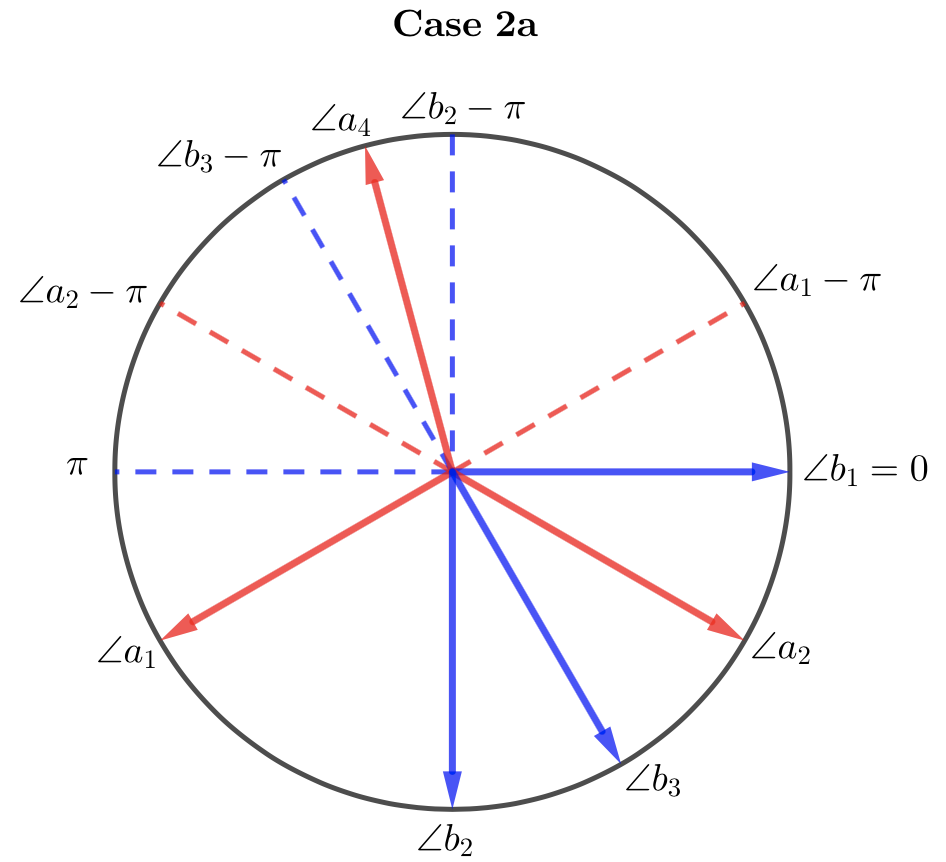}\hfill
  \includegraphics[width=.5\linewidth]{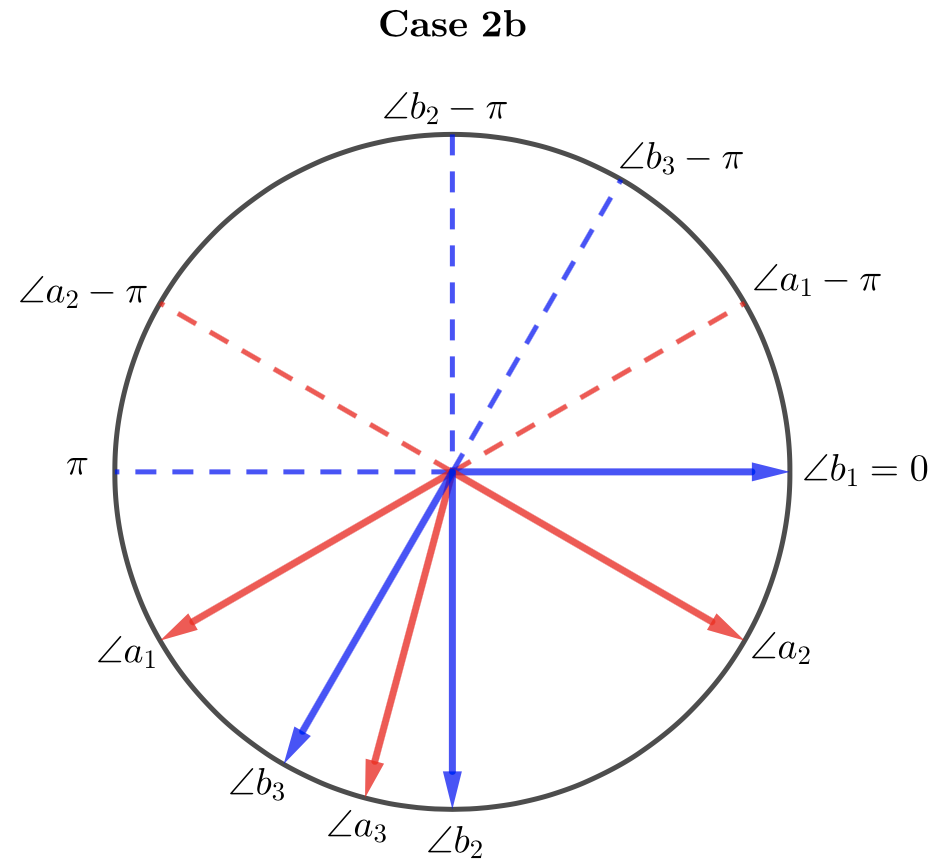}\hfill
  \end{subfigure}
  \caption{The ending states of each of the four cases in Example \ref{ex:pairwisecoherent_notrealizable}.}
    \label{fig:cases-pairwise-nonrealizable}
\end{figure}

\bibliography{references}{}
\bibliographystyle{amsalpha.bst}

\Addresses
\end{document}